\documentclass[a4paper,10pt]{amsart}
\usepackage{amsmath,amsthm,amssymb,latexsym,enumerate,color,hyperref}
\usepackage{graphicx}
\usepackage{color}
\usepackage{eucal}
\numberwithin{equation}{section}

\begin{document}

\newtheorem{thm}{Theorem}[section]
\newtheorem{prop}[thm]{Proposition}
\newtheorem{lem}[thm]{Lemma}
\newtheorem{cor}[thm]{Corollary}

\newtheorem{rem}[thm]{Remark}

\newtheorem*{defn}{Definition}

\newcommand{\DD}{\mathbb{D}}
\newcommand{\NN}{\mathbb{N}}
\newcommand{\ZZ}{\mathbb{Z}}
\newcommand{\QQ}{\mathbb{Q}}
\newcommand{\RR}{\mathbb{R}}
\newcommand{\CC}{\mathbb{C}}
\renewcommand{\SS}{\mathbb{S}}

\renewcommand{\theequation}{\arabic{section}.\arabic{equation}}

\newcommand{\Erfc}{\mathop{\mathrm{Erfc}}}    
\newcommand{\supp}{\mathop{\mathrm{supp}}}    
\newcommand{\re}{\mathop{\mathrm{Re}}}   
\newcommand{\im}{\mathop{\mathrm{Im}}}   
\newcommand{\dist}{\mathop{\mathrm{dist}}}  
\newcommand{\link}{\mathop{\circ\kern-.35em -}}
\newcommand{\spn}{\mathop{\mathrm{span}}}   
\newcommand{\ind}{\mathop{\mathrm{ind}}}   
\newcommand{\rank}{\mathop{\mathrm{rank}}}   
\newcommand{\ol}{\overline}
\newcommand{\pa}{\partial}
\newcommand{\ul}{\underline}
\newcommand{\diam}{\mathrm{diam}}
\newcommand{\lan}{\langle}
\newcommand{\ran}{\rangle}
\newcommand{\tr}{\mathop{\mathrm{tr}}}
\newcommand{\diag}{\mathop{\mathrm{diag}}}
\newcommand{\dv}{\mathop{\mathrm{div}}}
\newcommand{\na}{\nabla}
\newcommand{\nr}{\Vert}
\newcommand{\curl}{\mathop{\mathrm{curl}}}    

\newcommand{\al}{\alpha}
\newcommand{\be}{\beta}
\newcommand{\ga}{\gamma}  
\newcommand{\Ga}{\Gamma}
\newcommand{\de}{\delta}
\newcommand{\De}{\Delta}
\newcommand{\ve}{\varepsilon}
\newcommand{\fhi}{\varphi} 
\newcommand{\la}{\lambda}
\newcommand{\La}{\Lambda}    
\newcommand{\ka}{\kappa}
\newcommand{\vro}{\varrho}
\newcommand{\si}{\sigma}
\newcommand{\Si}{\Sigma}
\newcommand{\te}{\theta}
\newcommand{\zi}{\zeta}
\newcommand{\om}{\omega}
\newcommand{\Om}{\Omega}

\newcommand{\cA}{\mathcal{A}}
\newcommand{\cB}{\mathcal{B}}
\newcommand{\cC}{\mathcal{C}}
\newcommand{\cD}{\mathcal{D}}
\newcommand{\cE}{\mathcal{E}}
\newcommand{\cG}{{\mathcal G}}
\newcommand{\cH}{{\mathcal H}}
\newcommand{\cI}{{\mathcal I}}
\newcommand{\cJ}{{\mathcal J}}
\newcommand{\cK}{{\mathcal K}}
\newcommand{\cL}{{\mathcal L}}
\newcommand{\cM}{\mathcal{M}}
\newcommand{\cN}{{\mathcal N}}
\newcommand{\cP}{\mathcal{P}}
\newcommand{\cR}{{\mathcal R}}
\newcommand{\cS}{{\mathcal S}}
\newcommand{\cT}{{\mathcal T}}
\newcommand{\cU}{{\mathcal U}}
\newcommand{\cX}{\mathcal{X}}

\title[The interior Backus problem] {The interior Backus problem: \\
local resolution in H\"older spaces}

\author{Toru Kan} 
\address{Department of Mathematics, Osaka Metropolitan University,
1-1 Gakuen-cho, Naka-ku, Sakai, 599-8531, Japan.}
    \email{kan@omu.ac.jp}
    \urladdr{}

\author{Rolando Magnanini} 
\address{Dipartimento di Matematica ed Informatica ``U.~Dini'',
Universit\` a di Firenze, viale Morgagni 67/A, 50134 Firenze, Italy.}
    \email{magnanini@unifi.it}
    \urladdr{http://web.math.unifi.it/users/magnanin}

\author{Michiaki Onodera} 
\address{Department of Mathematics, Tokyo Institute of Technology, 
2-12-1 Ookayama, Meguro-ku, Tokyo 152-8551, Japan.}
    \email{onodera@math.titech.ac.jp}
    \urladdr{}

\begin{abstract}
We prove an existence result for the interior Backus problem in the Euclidean ball. The problem consists in determining a harmonic function in the ball from the knowledge of the modulus of its gradient on the boundary. The problem is severely nonlinear. From a physical point of view, the problem can be interpreted as the determination of the velocity potential of an incompressible and irrotational fluid inside the ball from measurements of the velocity field's modulus on the boundary. The linearized problem is an irregular oblique derivative problem, for which a phenomenon of loss of derivatives occurs. As a consequence, a solution by linearization of the Backus problem becomes problematic. Here, we linearize the problem around the vertical height solution and show that the loss of derivatives does not occur for solutions which are either (vertically) axially symmetric or oddly symmetric in the vertical direction. A standard fixed point argument is then feasible, based on ad hoc weighted estimates in H\"older spaces. 
\end{abstract}

\date{\today}

\keywords{Backus problem, fully nonlinear boundary conditions, irregular oblique derivative problem}
    \subjclass[2010]{35J65, 35C15, 35B07}

\maketitle

\raggedbottom

\section{Introduction}

Let $\Om$ be a bounded domain in the Euclidean space $\RR^N$, $N\ge 2$, with boundary $\Ga$. Let $g$ be a positive continuous function on $\Ga$. The \textit{interior Backus  problem} consists in determining a function $u\in C^1(\ol{\Om})\cap C^2(\Om)$ such that
\begin{equation}
\label{backus-problem}
\De u =0 \ \mbox{ in } \ \Om, \quad |\na u|=g \ \mbox{ on } \ \Ga.
\end{equation}
\par
This problem was first considered and completely solved in \cite{Ba1} (see also \cite{Ma}), for $N=2$. There, Backus was motivated by  a problem in geophysics, which entails the reconstruction of the gravitational or geomagnetic terrestrial field from measurements of its intensity on the Earth's surface $\cS$. In fact, if one models the Earth as the unit ball $B$, then the relevant geophysical problem amounts to determine solutions $u$ of \eqref{backus-problem}, with $\Om=\RR^N\setminus\ol{B}$ and $\Ga=\cS$, such that 
\begin{equation}
\label{decay-at-infinity}
u(x)\to 0 \ \mbox{ as } \ |x|\to\infty.
\end{equation}
This is what we call the \textit{exterior Backus problem}.
\par
When $N=2$, by the Riemann mapping theorem, we know that any simply connected (seen as a domain on the Riemann sphere) proper subdomain of the complex plane $\CC$  is conformally equivalent to the unit disk. Moreover, the harmonicity of functions is preserved by conformal mappings, while the modulus of their gradients changes just by a positive factor (which depends on the derivative of the conformal map). So, that is one reason why problem \eqref{backus-problem} has some interest.
\par
Another physical motivation, which genuinely pertains the interior problem setting, has to do with the study of incompressible and irrotational fluid flows. Let $\vec{V}$ be the velocity field of a fluid and let $\rho$ be its density. Any fluid flow obeys the continuity equation:
$$
\dv(\rho\,\vec{V})+\rho_t=0.
$$
If a fluid is incompressible, its density is constant. In particular, we have that $\rho_t=0$, and hence the continuity equation reduces to $\dv(\vec{V})=0$. If the fluid is irrotational, then $\curl\,(\vec{V})=0$, and hence there exists a harmonic velocity potential $u$ such that $\na u=\vec{V}$. Thus, solving  \eqref{backus-problem} can be interpreted as the determination of the velocity of the fluid inside the domain $\Om$ from measurements of its modulus on the boundary. 
\par
It is worthwhile to clarify from this point of view the results obtained in \cite{Ba1} for planar domains. As already mentioned, in this case we can always assume that $\Om$ is the unit disk. We shall describe the situation in the simplest case in which $g$ is assumed to be constant, say $g\equiv 1$. As shown in \cite{Ba1}, or by simply invoking \textit{Weierstrass factorization theorem} (see \cite{Ru}), in the complex variable $z=x+i\,y$, when $g \equiv 1$, the complex gradient $u_x-i\,u_y$ of $u$ is uniquely determined by the \textit{Blaschke product}
$$
u_x-i\,u_y=e^{i\,\al}\prod_{j=1}^n\frac{z-z_j}{1-\ol{z}_j\,z},
$$
where $\al\in\RR$ and $z_1, \dots, z_n\in\Om$ are given parameters. In fact, each factor in the product has unitary modulus on the unit circle. Thus, up to a rotation of an angle $\al$, the velocity field of the fluid can be uniquely determined from its modulus on the boundary if we know the position and the nature (the multiplicity, so as to speak) of its \textit{stagnation points} $z_1, \dots, z_n$.
\par
In order to conclude our motivations, it may be of interest to mention \cite{HHIKL}, in which it can be found a possible application to encephalography by magnetic means.

\smallskip

When $N\ge 3$, we can still use the Kelvin transformation to map the exterior of the ball to its interior (or an exterior domain to a bounded one) by preserving the harmonicity of functions. However, the boundary condition in \eqref{backus-problem} changes quite considerably. In fact, if 
$$
\cK w(y)=|y|^{2-N}\,w\left(\frac{y}{|y|^2}\right), \ \ y\ne 0,
$$
is the standard Kelvin transformation of a function $w:\RR^N\setminus\{0\}\to\RR$, we have that 
$$
\na \cK w(y)=|x|^N\left\{\na w(x)-\Bigl[2\,x\cdot\na w(x)+(N-2)\,w(x)\Bigr] \frac{x}{|x|^2}\right\},
$$
with $y=x/|x|^2$. In particular, if we apply the transformation to the exterior of unit ball $B$ centered at the origin, we obtain that the condition $|\na u|=g$ on $\cS$ changes into 
$$
|\na U+(N-2)\,U\,\nu|=g \ \mbox{ on } \ \cS,
$$
with $U=\cK u$. Here, $\nu$ denotes the exterior unit normal vector field on $\cS$.
\par
Another feature for which the interior and exterior problems differ from one another is that the latter admits solutions $u$ whose normal derivative $u_\nu$ does not change sign on $\cS$ (e.g. the fundamental solution or, more in general, the capacity potential of a bounded domain), while in the former the divergence theorem tells us that the mean value on $\cS$ of $u_\nu$ must be equal to zero. The positivity of the normal derivative of the solution has been useful in \cite{Jo} to obtain the local resolution of the exterior gravitational Backus problem for the Earth near the so-called \textit{monopole} $\Phi(x)=1/|x|$, based on a fixed-point argument. In other words, when $N=3$ and $\Om=\RR^3\setminus\ol{B}$, 
the existence and uniqueness of a solution $u$ of \eqref{backus-problem}-\eqref{decay-at-infinity} is obtained as the perturbation
$$
u=\Phi+v,
$$
where $v$ is harmonic in $\Om$ and decays to zero at infinity. This result holds if the data $g$ is sufficiently close to $1$ in a H\"older norm. (Notice that $1$ is the modulus of the gradient of the monopole on $\cS$.)
\par
We also mention that the positivity of $u_\nu$ is also used in \cite{DDO} to construct a comparison principle for suitably defined viscosity solutions for the exterior Backus problem and hence develop a nonlinear approach to the problem. 
\par
When the positivity property is not available, the only existence result up to date is given in \cite{KMO}. There, we consider, in physical dimesion $N=3$, the local resolution of the exterior Backus problem \eqref{backus-problem}-\eqref{decay-at-infinity} near the so-called \textit{dipole}:
$$
d(x)=\frac{x_3}{|x|^3}.
$$
The gradient $\na d$ models the  terrestrial  geomagnetic field. 
\par
The problem of finding solutions of \eqref{backus-problem}-\eqref{decay-at-infinity} of the type
$$
u=d+v,
$$
with $v$ harmonic in $\Om$, which decays to zero at infinity, has another level of difficulty, though. We can see this if we linearize \eqref{backus-problem}-\eqref{decay-at-infinity} near $d$. Indeed, we obtain the boundary value problem:
\begin{equation}
\label{oblique-dipole}
\De v=0 \ \mbox{ in } \ \Om, \quad \na d\cdot\na v=\fhi \ \mbox{ on } \ \cS, \quad v(x)\to 0 \ \mbox{ as } \ |x|\to\infty.
\end{equation}
This can be classified as an \textit{irregular} oblique derivative problem. In fact, in contrast with the monopole case, in which the vector field $\na\Phi$ governing the linearized problem is nothing else than the normal field on $\cS$, in \eqref{oblique-dipole}, instead, the field $\na d$
points outward to the Earth's surface on the southern hemisphere,
becomes tangential on the equator $\cE=\{x\in\cS : x_N = 0\}$, and points inward
on the northern hemisphere. For this reason, \eqref{oblique-dipole} suffers of two drawbacks. The first one is a severe lack of uniqueness, since its solutions can be uniquely determined only up to prescribing Dirichlet boundary values on $\cE$. The second one is a \textit{loss of regularity:} the expected solution $v$ does not gain the desired regularity. In other words, the regularity of $v$ does not improve that of the data $\varphi$ by one order --- it can be seen that it falls short of  $1/2$. This inconvenience makes the perturbation approach more complicated, because the iterative scheme on which a fixed-point argument is based upon \textit{loses derivatives} at each step. 
\par
In presence of a loss of derivatives, the Nash-Moser implicit function theorem has worked in other contexts (see the pioneering works \cite{Na, Mo1, Mo2, Ha}, for instance). Unfortunately, in attacking the Backus problem, this plan is so far out of reach. In fact, one would need sufficiently precise estimates for the relevant oblique derivative problems. Namely, 
it is necessary to have an accurate control not only for the solution of \eqref{oblique-dipole},
but also for those of a class of oblique derivative problems obtained by perturbing $\na d$.
\par
Nevertheless, in \cite{KMO} we showed that \eqref{backus-problem}-\eqref{decay-at-infinity} is solvable near $d$ for solutions which are \textit{axially symmetric} around the Earth's axis $\cA=\{ \la\,(0,0,1): \la\in\RR\}$. In fact, we show that the solutions of the linearized problem \eqref{oblique-dipole} with this symmetry no longer lose derivatives in an appropriate scale of fractional Sobolev spaces on $\cS$. This result is made possible by the use of spherical harmonics on $\cS$. As a consequence, the relevant fixed-point scheme can be mended and the existence of a solution of \eqref{backus-problem}-\eqref{decay-at-infinity} is obtained if $g$ is sufficiently close to $|\na d|$ in some fractional Sobolev norm.
\par
In the present paper, for $N\ge 3$, we turn our attention to the local resolution of the interior Backus problem  \eqref{backus-problem} in the framework of H\"older spaces. This framework is that used in 
\cite{Jo} for the gravitational case. The simplest instance in this case is to consider solutions of \eqref{backus-problem} in the form:
$$
u(x)=x_N+v(x),
$$
where $v$ is harmonic in $\Om$. If we place the $x_N$-axis horizontally, in the fluidmechanical framework mentioned above, we want to determine the velocity of the fluid inside a domain from measurements of its modulus on the boundary as a perturbation of that of an horizontal laminar flow with potential 
$$
f(x)=x_N.
$$
\par
In this case, the associated linearized problem is simply:
\begin{equation}
\label{oblique-laminar}
\De v =0 \ \mbox{ in } \ \Om, \quad  \pa_{x_N} v=\fhi \ \mbox{ on } \ \cS,
\end{equation}
where $\varphi$ is a given function.
It is clear that the vector field governing \eqref{oblique-laminar} is $e_N=(0,\dots, 0, 1)$, that shows similar qualitative features to those of $\na d$, in the sense that $-e_N$ points inward on the northern hemisphere, is tangential on $\cE$, and points outward on the southern hemisphere. Also in this case, though, a loss of derivatives occurs for problem \eqref{oblique-laminar}. Nevertheless, we shall see that the somewhat easier oblique derivative condition in \eqref{oblique-laminar} allows a treatment of \eqref{backus-problem} in the framework of H\"older spaces, 
in the cases where solutions are oddly symmetric with respect to the hyperplane $x_N=0$
or axially symmetric around $x_N$-axis.
\par
In order to describe the main result of this paper, we need to introduce some notation.
For $k=0, 1, 2, \dots$ and $\al \in (0,1)$, 
we define the function spaces:
\begin{gather*}
C_{\rm even}^{k+\al}(\ol{B})
 =\{ \fhi \in C^{k+\al}(\ol{B}): \fhi(x',x_N)=\fhi(x',-x_N)\},
\\
C_{\rm odd}^{k+\al}(\ol{B})
 =\{ \fhi \in C^{k+\al}(\ol{B}): \fhi(x',x_N)=-\fhi(x',-x_N)\},
\\
C_{\rm ax}^{k+\al}(\ol{B})
 =\{ \fhi \in C^{k+\al}(\ol{B}): \fhi(x',x_N)=\fhi(|x'|e_1',x_N) \}.
\end{gather*}
Here $e_1'=(1,0,\ldots,0) \in \RR^{N-1}$.
We note that these are closed subspaces of the space $C^{k+\al}(\ol{B})$ of $k$-differentiable functions whose derivatives up to the order $k$ are $\al$-H\"older continuous.
The usual H\"{o}lder seminorm and norm on $C^{k+\al}(\ol{B})$ 
are denoted by $[\cdot]_{\al,\ol{B}}$ and $|\cdot|_{k+\al,\ol{B}}$, respectively, and defined as:
\begin{gather*}
[u]_{\al,\ol{B}}=\sup_{x,y \in \ol{B}, \, x \neq y} \frac{|u(x)-u(y)|}{|x-y|^\al},
\\
|u|_{k+\al,\ol{B}}=\sum_{|\ga| \le k} |D^\ga u|_{0,\ol{B}}
 +\sum_{|\ga|=k} [D^\ga u]_{\al,\ol{B}}.
\end{gather*}
Here, $|\cdot|_{0,\ol{B}}$ stands for the standard (uniform) maximum norm.
\par
With these premises, 
we can state the main result of this paper as follows.

\begin{thm}
\label{thm:existence0}
Let $\al \in (0,1)$ and set $\Om=B$.
Then there exist positive constants $\de_0$ and $C$ with the following properties.
\begin{itemize}
\item[(i)]
If $g \in C^{1+\al}_{\rm even}(\ol{B})$ and
\begin{equation*}
|g -1|_{1+\al,\ol{B}} \le \de_0,
\end{equation*}
then problem \eqref{backus-problem} has a solution 
$u \in C^2(B) \cap C^{1+\al}_{\rm odd}(\ol{B})$ satisfying
\begin{equation*}
|u-f|_{1+\al,\ol{B}} \le C\,|g -1|_{1+\al,\ol{B}}.
\end{equation*}
\item[(ii)]
If $g \in C^{1+\al}_{\rm ax}(\ol{B})$, $h \in \RR$ and
\begin{equation*}
|g -1|_{1+\al,\ol{B}} \le \de_0,
\end{equation*}
then problem \eqref{backus-problem} has a solution 
$u \in C^2(B) \cap C^{1+\al}_{\rm ax}(\ol{B})$ satisfying
\begin{equation*}
u=h \ \mbox{ on } \ \cE,
\end{equation*}
and
\begin{equation*}
|u-h-f|_{1+\al,\ol{B}} \le C\,\left( |g -1|_{1+\al,\ol{B}}\right).
\end{equation*}
\end{itemize}
\end{thm}

Note that the solution $u$ obtained in (i) of Theorem~\ref{thm:existence0} 
automatically satisfies the condition $u=0$ on $\cE$,
since it is odd in the variable $x_N$; 
while in (ii) we need to impose the additional boundary condition on $\cE$, due to the fact that \eqref{backus-problem} is invariant under the addition of constants. 
\par
It is worthwhile noticing at this point that, even if the height $f$ and the dipole $d$ are the Kelvin trasformations of one another, $\cK$ maps the linearized exterior problem \eqref{oblique-dipole} into the  interior oblique derivative problem:
$$
\De v=0 \ \mbox{ in } \ B, \quad \pa_{x_N} v +(N-2)\,f\,\pa_{\nu} v+(N-1)(N-2)\,f\,v=\fhi \ \mbox{ on } \ \cS.
$$
This problem is more difficult to treat, even if the relevant vector field governing it has the same qualitative properties of $e_N$. A study of this problem will be the theme of future work.
\par
The proof of Theorem \ref{thm:existence0} hinges on some a priori estimates for the solution of the linearized problem \eqref{oblique-laminar} subject to the Dirichlet-type condition
\begin{equation}
\label{equator}
v=\psi \ \mbox{ on } \ \cE.
\end{equation}
Thus, in Section \ref{sec:alimov}, we will first derive an explicit representation formula for the solutions of \eqref{oblique-laminar}-\eqref{equator}. We will also adapt a couple of lemmas in \cite{Al} to our purposes, by making explicit the dependence of the relevant norms of $v$ on those of the data $\fhi$.
Then, in Section \ref{sec:a-priori-estimates}, we shall derive crucial a priori estimates for the linearized problem \eqref{oblique-laminar}-\eqref{equator} (see Theorem \ref{prop:estimate-oblique-interior0}). 
Finally, in Section \ref{sec:backus}, based on these estimates, we will carry out the proof of Theorem \ref{thm:existence0}.
\section{An explicit integral representation formula \\ for the oblique derivative problem}
\label{sec:alimov}
In this section, we carry out explicit computations which lead to the construction of an integral representation formula for the linearized problem \eqref{oblique-laminar}-\eqref{equator}. We follow the scheme introduced in \cite{Al}.

\subsection{Uniqueness for problem \eqref{oblique-laminar}-\eqref{equator}}

Let a function space $C^1_N(\ol{B})$ be defined by
\begin{equation*}
C^1_N(\ol{B})=\{ v \in C(\ol{B}): \pa_{x_N} v \mbox{ exists and } \pa_{x_N} v \in C(\ol{B})\}.
\end{equation*}
Then $C^1_N(\ol{B})\cap C^2(B)$ is one of the natural spaces 
for solutions of \eqref{oblique-laminar}-\eqref{equator}.
We show that a uniqueness result holds in this space.
The argument follows the lines of one used in \cite{JM} and is based on Hopf's boundary lemma. 
\begin{prop}[Uniqueness]
\label{prop:uniqueness}
For any given $\fhi\in C(\cS)$ and $\psi\in C(\cE)$, 
the problem \eqref{oblique-laminar}-\eqref{equator} has at most one solution 
of class $C^1_N(\ol{B})\cap C^2(B)$.
\end{prop}
\begin{proof}
Let $v_1$ and $v_2$ be two solutions of class $C^1_N(\ol{B})\cap C^2(B)$ of \eqref{oblique-laminar}-\eqref{equator}, and set $v=v_1-v_2$. Then $v$ solves \eqref{oblique-laminar}-\eqref{equator} with $\fhi \equiv 0$ and $\psi \equiv 0$. If $v(x)$ is a (positive) maximum for $v$ on $\ol{B}$, then we have that $x\in\cS$, by the maximum principle, and $x\notin\cE$, being as $v(x)>0$.
\par
Now, if $x$ were in the upper hemisphere of $\cS$, then we would have that $\partial_{x_N} v(x)>0$, by Hopf's boundary lemma, since $e_N$ is an outward direction on that hemisphere. This is a contradiction. By a similar argument, we infer that $x$ cannot belong to the lower hemisphere of $\cS$. Thus, we conclude that $v\le 0$ on $\ol{B}$. By considering the minimum of $v$, we can infer that $v\ge 0$ on $\ol{B}$, and hence $v\equiv 0$ on $\ol{B}$.
\end{proof}

We can easily use this proposition to infer that the solution $v$ of  \eqref{oblique-laminar}-\eqref{equator} inherits possible symmetries of the data $\fhi$ and $\psi$. 
For instance, if $\fhi$ and $\psi$ are axially symmetric around $x_N$-axis, then $v$ is so.

\subsection{Estimates for the Dirichlet problem for the Laplace equation}

In the next result,
we derive estimates for solutions of the Dirichlet problem in $B$:
\begin{equation}
\label{Laplace-Dirichlet}
\De w=0 \ \mbox{ in } \ B, \quad w=\fhi \ \mbox{ on } \ \cS.
\end{equation}
It is well-known that 
if $\fhi\in C(\cS)$, this problem has the unique solution $w \in C(\ol{B})\cap C^2(B)$
given by the \textit{Poisson integral formula}:
\begin{equation}
\label{def-w}
w(x)=\int_\cS P_B(x;y)\,\fhi(y)\,dS_y.
\end{equation}
Here, $P_\Om$ stands for the Poisson kernel for a bounded domain $\Om$.
In particular, $P_B$ is explicitly given by
\begin{equation*}
P_B(x;y)=\frac1{N\om_N}  \frac{1-|x|^2}{|x-y|^N} \ \mbox{ for } \ x\in B, \ y\in\cS, 
\end{equation*}
where $\om_N$ is the volume of $B$. 
For our aims, we need the following refinement of \cite[Lemma 2.2]{Al}.

\begin{prop}\label{lem:estimate-derivative-w}
Suppose that $\fhi\in C^{k+\al}(\ol{B})$ 
for some non-negative integer $k$ and $\al \in [0,1)$. 
Let $\be$ be a multi-index with $|\be|>k+\al$.
\par
If $w$ is the solution of \eqref{Laplace-Dirichlet}, 
then there exists a positive constant $C=C(N,\be,k,\al)$ such that 
\begin{equation}
\label{derivatives-growth}
|D^\be w(x)| \le C\,|\fhi|_{k+\al,\ol{B}}\,(1-|x|^2)^{-|\be|+k+\al}
\end{equation}
for all $x \in B$.
\end{prop}

To prove Proposition~\ref{lem:estimate-derivative-w},
we first recall the following a priori estimate for the Laplace equation.
For the proofs, see for instance \cite[Theorem~6.6, Problem~6.2]{GT}.
\begin{lem}\label{lem:estimate-Poisson-integral}
Suppose that $\fhi$ is of class $C^{k+\al}(\ol{B})$ 
for some integer $k \ge 2$ and $\al \in (0,1)$. 
Then the solution $w$ of \eqref{Laplace-Dirichlet} satisfies the inequality
\begin{equation*}
|w|_{k+\al,\ol{B}} \le C|\varphi|_{k+\al,\ol{B}}
\end{equation*}
for some positive constant $C=C(N,k,\al)$.
\end{lem}

Next, we derive a pointwise estimate for the Poisson kernel.
\begin{lem}
\label{prop:derivatives-P}
For any multi-index $\be$, it holds that 
$$
D^\be_x P_B(x;y)=\frac{a_\be(x,y)}{|x-y|^{|\be|+N-1}} \ \mbox{ for } \ x\in B, \ y\in \cS,
$$
where $|a_\be(x,y)|\le C_*$ for some positive constant $C_*$ which only depends on $N$ and $\be$.
\end{lem}

\begin{proof}
We have that 
$$
P_B(x;y)=\frac1{N\om_N}  \frac{1-|x|^2}{|x-y|^N}=\frac1{N\om_N}  \frac{|y|^2-|x|^2}{|x-y|^N}=\frac1{N\om_N}  \frac{(x-y)\cdot(x+y)}{|x-y|^N}.
$$
Hence, if we set $z=x-y$, we obtain that
$$
P_B(z+y;y)=\frac1{N\om_N}  \frac1{|z|^{N-2}}+\frac2{N\om_N}\,\frac{y\cdot z}{|z|^N}.
$$
This function of $z$ is the sum of a $(2-N)$-homogeneous and a $(1-N)$-homogeneous function. Thus, we infer that 
$$
D^{\be}_zP_B(z+y;y)=A(z; y)+B(z; y),
$$ 
where $A(z; y)$ and $B(z; y)$ are homogeneous of degree $2-N-|\be|$ and $1-N-|\be|$ in $z$. As a consequence, we get:
$$
A(z; y)+B(z; y)=|z|^{1-N-|\be|}\bigl[|z|\,A(z/|z|;y)+B(z/|z|;y)\bigr].
$$
The function in the brackets is bounded since both $z/|z|$ and $y$ have a unitary norm and $z\in 2B$. Therefore, we conclude by setting
$$
a_\be(x,y)=|x-y|\,A\left(\frac{x-y}{|x-y|};y\right)+B\left(\frac{x-y}{|x-y|};y\right),
$$
for $x\in B$ and $y\in\cS$.
\end{proof}

We also need the following bound.
\begin{lem}
\label{lem:integral-P}
Let a multi-index $\be$ and a nonnegative number $\ka$ satisfy $|\be|>\ka$.
Then, there exists a positive constant $C=C(N,\be,\ka)$ such that
$$
\int_{\cS} |D^\be_x P_B(x;y)| \, |y-x_0|^\ka \, dS_y \le C(1-|x|)^{-|\be|+\ka},
$$
for all $x \in B\setminus\{ 0\}$, where $x_0=x/|x| \in \cS$. When $x=0$, we can choose $x_0$ to be any point in $\cS$.
\end{lem}

\begin{proof}
When $x=0$, the bound easily follows from Lemma~\ref{prop:derivatives-P}. Let $x \in B\setminus\{ 0\}$, $x_0=x/|x|$, 
and set $r=1-|x|$.
Then, for every $y \in \cS$, we have that
\begin{multline*}
|y-x|=\frac{2}{3}|y-x| +\frac{1}{3}|(y-x_0) -(x-x_0)| \ge
\\ 
\frac{2}{3} (|y|-|x|) +\frac{1}{3}(|y-x_0| -|x-x_0|)
 =\frac{1}{3}r +\frac{1}{3} |y-x_0|.
\end{multline*}
This with Lemma~\ref{prop:derivatives-P} shows that
\begin{multline*}
\int_{\cS} |D^\be_x P_B(x;y)| \, |y-x_0|^\ka \, dS_y
 =\int_{\cS} \frac{|a_\be(x,y)||y-x_0|^\ka}{|y-x|^{|\be|+N-1}} \, dS_y \le 
\\
3^{1-N-|\be|}C_* \int_\cS \frac{|y-x_0|^{\ka}}{(r+|y-x_0|)^{|\be|+N-1}}\,dS_y =
\\
3^{1-N-|\be|}C_* r^{-|\be|+\ka} \int_{|rz +e_N| =1}
 \frac{|z|^{\ka}}{(1+|z|)^{|\be|+N-1}}\,dS_z,
\end{multline*}
where we have used the change of variables $y=x_0+r \cR z$ 
with the orthogonal matrix $\cR$ satisfying $\cR^{-1} x_0=e_N=(0,\ldots,0,1)$.
The last integral on the right-hand side of the above inequality is bounded with respect to $r$,
because as $r \to 0$ it converges to the integral
\begin{equation*}
\int_{\RR^{N-1}} \frac{|z|^{\ka}}{(1+|z|)^{|\be|+N-1}}\,dS_z,
\end{equation*}
which is finite, being as $|\be|>\ka$.
We thus obtain the desired inequality.
\end{proof}


\begin{proof}[Proof of Proposition~\ref{lem:estimate-derivative-w}]
As usual, in this proof $C$ will denote a generic constant possibly depending on $N, \be, k$, and $\al$.
\par
Pick any point $x_0 \in \cS$.
Since $\fhi\in C^{k+\al}(\ol{B})$, we can write the following standard Taylor expansion for $\fhi$:
$$
\fhi(y)=\sum_{j=0}^k \sum_{|\ga|=j} \frac{D^\ga\fhi(x_0)}{\ga!}\,(y-x_0)^\ga+
\sum_{|\ga|=k}\frac{D^\ga\fhi(x_0+\te\,(y-x_0))-D^\ga\fhi(x_0)}{\ga!}\,(y-x_0)^\ga.
$$
Here, we use the standard conventions on the multi-index notation. Thus, integrating $\fhi(y)$ for $y\in\cS$ against $D^\be_x P_B(x;y)$ gives that
\begin{equation}\label{w-equality}
D^\be w(x)=\sum_{|\ga| \le k} \frac{D^\ga\fhi(x_0)}{\gamma !}\,D^\be h_{\ga}(x)+R_k(x),
\end{equation}
where $h_{\ga}$ is the solution of \eqref{Laplace-Dirichlet} with $\fhi=(\cdot-x_0)^\ga$ and
$$
R_k(x)= \sum_{|\ga|=k}\int_\cS D^\be_x P_B(x;y)\,\frac{D^\ga\fhi(x_0+\te\,(y-x_0))-D^\ga\fhi(x_0)}{\ga!}\,(y-x_0)^\ga\,dS_y.
$$
\par
Now, let $x_*$ be any point in $B \setminus \{0\}$ such that $x_*=|x_*|x_0$.
Then, Lemma~\ref{lem:estimate-Poisson-integral} gives that
\begin{equation}\label{h-inequality}
|D^\be h_{\ga}(x_*)| \le |D^\be h_{\ga}|_{0,\ol{B}} \le C.
\end{equation}
Moreover, if $\al \in (0,1)$, Lemma~\ref{lem:integral-P} shows that
\begin{multline}\label{Rk-inequality}
|R_k(x_*)| \le \sum_{|\ga|=k} \frac{[D^\ga \fhi]_{\al,\ol{B}}}{\ga!}
 \int_\cS |D^\be_x P_B(x_*;y)| \, |y-x_0|^{k+\al} \,dS_y \le 
\\ 
C\sum_{|\ga|=k} [D^\ga \fhi]_{\al,\ol{B}} \, (1-|x_*|)^{-|\be|+k+\al}.
\end{multline} 
This inequality is also valid for $\al=0$,
if $[D^\ga \fhi]_{\al,\ol{B}}$ is replaced by $|D^\ga \fhi|_{0,\ol{B}}$.
Plugging \eqref{h-inequality} and \eqref{Rk-inequality} into \eqref{w-equality},
then gives that 
$$
|D^\be w(x_*)|\le C\,|\fhi|_{k+\al,\ol{B}}\,(1-|x_*|)^{-|\be|+k+\al}.
$$ 
which yields \eqref{derivatives-growth}, after an update of the constant $C$.
\end{proof}


\subsection{Representation formulas for problem \eqref{oblique-laminar}-\eqref{equator}}

In order to obtain a representation formula,
we consider the Dirichlet problem
\begin{equation}
\label{eq-for-Z}
-\De_{x'} Z(x')=\pa_{x_N} w(x',0) \ \mbox{ in } \ D,
\quad Z=\psi \ \mbox{ on } \ \pa D,
\end{equation}
where $w$ is the solution of \eqref{Laplace-Dirichlet} and $D=\{ x' \in \RR^{N-1}: |x'|<1\}$.
We identify $D$ and $\pa D$ 
with the equatorial ball $\{x=(x',x_N)\in B: x_N=0\}$ and the equator $\cE$, respectively.
From Lemma~\ref{lem:integral-P},
we see that $w$ satisfies
\begin{equation*}
|\pa_{x_N} w(x',0)| \le |\fhi|_{0,\cS} \int_{\cS} |\pa_{x_N} P_B(x',0; y)| \, dS_y
 \le C\,|\fhi|_{0,\cS} \, (1-|x'|)^{-1}
\end{equation*}
for some constant $C$.
Therefore, for any $\fhi \in C(\cS)$ and $\psi \in C(\cE)$,
the existence and uniqueness of solutions of \eqref{eq-for-Z} in $C(\ol{D}) \cap C^2(D)$
are guaranteed by \cite[Theorem~4.9]{GT}.

\begin{prop}[Existence and representation formula]
\label{prop:representation0}
Suppose that $\fhi \in C(\cS)$ and $\psi \in C(\cE)$.
Let $w$ and $Z$ be the solutions of \eqref{Laplace-Dirichlet} and \eqref{eq-for-Z}, respectively, 
and set
\begin{equation}
\label{def-W}
W(x)=\int_0^{x_N} w(x',t)\,dt, \quad x=(x', x_N)\in \ol{B},
\end{equation}
where $w$ is defined in \eqref{def-w}.
Then the unique solution $v$ of class $C^1_N(\ol{B})\cap C^2(B)$ of \eqref{oblique-laminar}-\eqref{equator} is given by
\begin{equation}\label{u-representation0}
v(x)=W(x)+Z(x'), \quad x=(x', x_N)\in \ol{B}.
\end{equation}
\end{prop}

\begin{proof}
Let $v$ be defined by \eqref{u-representation0}.
Then $v=W+Z \in C(\ol{B}) \cap C^2(B)$,
since we know that $w \in C(\ol{B}) \cap C^2(B)$ and $Z \in C(\ol{D}) \cap C^2(D)$.
Moreover, we have that $\pa_{x_N} v=w \in C(\ol{B})$,
and therefore $v \in C^1_N(\ol{B})\cap C^2(B)$.

Since $W(x',0)=0$,
we see that 
\begin{equation*}
\pa_{x_N} v=w=\fhi \ \mbox{ on } \ \cS,
\quad
v=Z=\psi \ \mbox{ on } \ \cE.
\end{equation*}
Hence the assertion follows if we show that $v$ is harmonic in $B$.
By a direct calculation,
we have that
\begin{multline*}
\De W(x)=\int_0^{x_N} \De_{x'} w(x',t)\,dt+\pa_{x_N}w(x',x_N)= \\
-\int_0^{x_N} \pa_{x_N x_N}^2w(x',t)\,dt+\pa_{x_N}w(x',x_N)=\pa_{x_N}w(x',0).
\end{multline*}
We thus infer that 
\begin{equation*}
\De v(x)=\De W(x) +\De_{x'} Z(x') =\pa_{x_N}w(x',0) -\pa_{x_N}w(x',0)=0,
\end{equation*}
as desired.
\end{proof}

Even if it is not needed in the proof of Theorem \ref{thm:existence0}, we also derive for future reference an explicit integral representation formula. 
The formula may be helpful for numerical approximations. To derive the formula, we recall that the fundamental solution $\Ga_d$ of the Laplace equation in a $d$-dimensional Euclidean 
space ($d\ge 2$) is given by
\begin{equation*}
\Ga_2(x)=\frac1{2\pi}\,\log\frac{1}{|x|}, 
\qquad
\Ga_d(x)= \frac1{d(d-2)\,\om_{d}}|x|^{2-d} \ \mbox{ if } \ d\ge 3.
\end{equation*}
Then, the Green's function for $D$ is written as
\begin{equation*}
G_D(x'; y')=\Ga_{N-1}(x'-y') -\Ga_{N-1}\left( |x'|(\cI(x')-y')\right),
\end{equation*}
where $\cI$ denotes the inversion $\cI(x')=x'/|x'|^2$ for $x'\ne 0$.
\par
If we now define the kernel
\begin{equation*}
\label{def-K-alimov}
K(x;y)=\int_0^{x_N} P_B(x',t; y)\,dt+ 
\int_D G_D(x',z')\, \pa_{x_N} P_B(z',0;y)\,dz',
\end{equation*}
for $x=(x', x_N)\in B$ and $y\in\cS$, 
the representation formula is given as follows.

\begin{prop}[Integral representation formula]
\label{th:representation}
\par
Suppose that $\fhi\in C(\cS)$ and $\psi\in C(\cE)$. Then, the function defined by 
\begin{equation}
\label{representation-alimov}
v(x)=\int_\cS K(x;y)\, \fhi(y)\,dS_y+\int_{\cE} P_D(x';y')\,\psi(y')\,dS_{y'} \ \mbox{ for } \ x\in B
\end{equation}
coincides with the unique solution of class $C^1_N(\ol{B})\cap C^2(B)$ of the problem \eqref{oblique-laminar}-\eqref{equator}.
\end{prop}

\begin{proof}
By the well-known representation formula for the Dirichlet problem for the Poisson equation,  we know that
$$
Z(x')=\int_D G_D(x';z')\,\pa_{x_N}w(z',0)\,dz'+\int_\cE P_D(x';z')\,\psi(z')\,dS_{z'}.
$$
With the definition of $w$ in mind, by the Fubini theorem we then infer that
$$
\int_D G_D(x';z')\,\pa_{x_N}w(z',0)\,dz'=\int_S \left[\int_D G_D(x';z')\,\pa_{x_N}P_B(z',0;y)\,dz'\right]\fhi(y)\,dS_y.
$$
Being as $x'\in D$, the Fubini theorem is applicable in this formula, because the function $F$ defined a.e. on $D\times\cS$ by  $F(z',y)=\fhi(y)\,\pa_{x_N}P_B(z',0;y)\,G_D(x';z')$ is in $L^1(D\times\cS)$. In fact, we have that
$$
\int_{D\times\cS}|F(z',y)|(dz'\times dS_y)=\int_D G_D(x';z') \left[ \int_\cS |\pa_{x_N}P_B(z',0;y)||\fhi(y)|\,dS_y \right] dz'
$$
(see \cite[Theorem 1.12]{LL}). The right-hand side is finite thanks to the properties of $G_D$ and  Lemma \ref{lem:integral-P} with $\ka=0$ and $|\be|=1$.
\par
Finally, that 
$$
W(x)=\int_\cS \left[\int_0^{x_N} P_B(x',t; y)\,dt\right]\fhi(y) \, dS_y
$$
follows from \eqref{def-W}, again by a straightforward application of the Fubini theorem.
We have thus proved that the right-hand side of \eqref{representation-alimov} coincides with $W+Z$.
\end{proof}

\section{A priori estimates for the linearized problem}
\label{sec:a-priori-estimates}

A crucial a priori bound we will use to prove Theorem \ref{thm:existence0} is contained in the following theorem.

\begin{thm}\label{prop:estimate-oblique-interior0}
Let $\al \in (0,1)$ and suppose that 
$\fhi\in C^{1+\al}(\ol{B})$ and $\psi \in C^{3/2+\al}(\ol{D})$.
Then a solution $v$ of \eqref{oblique-laminar}-\eqref{equator} has the properties
\begin{equation*}
v \in C^{1+\al}(\ol{B}), 
\quad \partial_{x_N} v  \in C^{1+\al}(\ol{B}), 
\quad x_N D^2_{x'} v \in C^{\al}(\ol{B}).
\end{equation*}
Moreover, the following inequality holds 
for some positive constant $C$ independent of $\fhi$ and $\psi$:
\begin{equation}\label{estimate-oblique-interior}
|v|_{1+\al,\ol{B}} +|\partial_{x_N} v|_{1+\al,\ol{B}}
 +|x_N D^2_{x'} v|_{\al,\ol{B}} \le C\left( |\fhi|_{1+\al,\ol{B}} +|\psi|_{3/2+\al,\ol{D}} \right).
\end{equation}
\end{thm}

We will use the following simple estimate,
which is a refinement of \cite[Lemma 3.2]{Al}.
\begin{lem}\label{lem:estimate-integral0}
Let $\ka>0$.
Then there exists a constant $C>0$ such that
\begin{equation*}
|x_N|\int_0^{|x_N|} \frac{dt}{(1-|x'|^2-t^2)^{1+\ka}} \le \frac{C}{(1-|x|^2)^{\ka}}
\end{equation*}
for all $x \in B$.
\end{lem}

\begin{proof}
Set $|x_N|=\si\,\sqrt{1-|x'|^2}$; it holds that $0\le\si<1$ for $x=(x', x_N)\in B$. By the change of variable $t=s\,\sqrt{1-|x'|^2}$, we have that
$$
\int_0^{|x_N|} \frac{dt}{(1-|x'|^2-t^2)^{1+\ka}} 
 =\frac1{(1-|x'|^2)^{1/2+\ka}} \int_0^\si \frac{ds}{(1-s^2)^{1+\ka}} ,
$$
and hence
$$
|x_N| (1-|x|^2)^{\ka}\int_0^{|x_N|} \frac{dt}{(1-|x'|^2-t^2)^{1+\ka}} =\si (1-\si^2)^{\ka}
\int_0^\si \frac{ds}{(1-s^2)^{1+\ka}}.
$$
The right-hand side is bounded by some constant $C$, 
since L'H\^{o}pital's rule shows that its limit as $\si\to 1^-$ is equal to $1/(2\ka)$.
Thus the lemma follows.
\end{proof}

The following lemma is essentially shown in \cite[Lemma 2.5]{Al}.
For the sake of completeness, we give a proof.

\begin{lem}
\label{lem:estimate-holder-seminorm}
Let $v \in C^1(B) \cap C(\ol{B})$. Suppose that there exist a positive constant $M$ and an  exponent $\al \in (0,1)$ such that
\begin{equation*}
|\nabla v(x)| \le M(1-|x|^2)^{-1+\al} \ \mbox{ for all } \ x\in B.
\end{equation*}
Then $v \in C^\al(\ol{B})$ and it holds that
\begin{equation}
\label{a}
[v]_{\al,\ol{B}} \le CM,
\end{equation}
for some positive constant $C$ only depending on $\al$.
\end{lem}

\begin{proof}
The assumption on $v$ gives that
\begin{equation}
\label{ineq-grad}
|\na v(x)|\le M\,(1-|x|)^{-1+\al}=M\,|x-\ol{x}|^{-1+\al} \ \mbox{ with } \ \ol{x}=x/|x|,
\end{equation}
since $\alpha<1$.
\par
(i) Let $\vartheta=|x-\ol{x}|$ and $\ell=(x-\ol{x})/\vartheta=-\ol{x}$. Then, we have that
$$
|v(x)-v(\ol{x})|=\left|\int_0^\vartheta \na v(\ol{x}+t\,\ell)\cdot\ell\,dt\right| \le 
M\,\int_0^\vartheta t^{-1+\al}dt =\frac{M}{\al} \vartheta^{\al}.
$$
\par
(ii) Let $\ol{x}$ and $\ol{y}$ be arbitrary points on $\cS$ with $\vartheta=|\ol{x}-\ol{y}|<1$. Take  $x=(1-\vartheta)\,\ol{x}$ and $y=(1-\vartheta)\,\ol{y}$. Then, for an intermediate point $\xi$ between $x$ and $y$, we have that
\begin{multline*}
|v(\ol{x})-v(\ol{y})|\le |v(x)-v(\ol{x})|+|v(x)-v(y)|+|v(y)-v(\ol{y})| \le \\
\frac{2\,M}{\al} \vartheta^\al+|\na v(\xi)|\,|x-y|\le \frac{2\,M}{\al}\vartheta^\al+M\,\vartheta^{-1+\al} |x-y|\le \\ \frac{2\,M}{\al} \vartheta^\al+M\,\vartheta^\al.
\end{multline*}
Here, we have used (i), \eqref{ineq-grad} and the fact that $|\xi-\ol{\xi}| \ge |x-\ol{x}|=\vartheta$. 
\par
If $\vartheta=|\ol{x}-\ol{y}|\geq 1$, then one can choose a finite number of points on $\cS$, say $\ol{x}=\ol{x}_0,\ol{x}_1,\ol{x}_2,\ol{x}_3,\ol{x}_4=\ol{y}$ with $|\ol{x}_i-\ol{x}_{i+1}|<1$ so that the previous estimate applies, and the combination of the estimates yields the desired inequality. 
Therefore, $|v(\ol{x})-v(\ol{y})|\le C M\,|\ol{x}-\ol{y}|^\al$, for some constant $C$.
\par
(iii) Now, take $x, y\in B$, set $\vartheta=|x-y|$, and let $\ol{x}$ and $\ol{y}$ be the usual projections of $x$ and $y$ on $\cS$. We can always assume that $|y-\ol{y}|\ge |x-\ol{x}|$.
If $|x-\ol{x}|\ge \vartheta$, then, for an intermediate point $\xi$ between $x$ and $y$, 
$$
|v(x)-v(y)|\le |\na v(\xi)|\,|x-y|\le M\,\vartheta^{-1+\al}|x-y|=M\,\vartheta^\al,
$$
thanks to the inequalities $|\xi-\ol{\xi}| \ge |x-\ol{x}| \ge \vartheta$. 
If $|x-\ol{x}|< \vartheta$ instead, we first infer that 
$$
|y-\ol{y}|\le |y-\ol{x}|\le |y-x|+ |x-\ol{x}|<2\,\vartheta.
$$ 
Thus, (i) gives that $|v(x)-v(\ol{x})|\le M\,\vartheta^\al/\al$ and  $|v(y)-v(\ol{y})|\le 2^\al \, M\, \vartheta^\al/\al$, while (ii) yields:
$$
|v(\ol{x})-v(\ol{y})|\le C M\,|\ol{x}-\ol{y}|^\al\le 4^\al \, C M\, \vartheta^\al.
$$
We then conclude thanks to the triangle inequality.
The bound \eqref{a} then follows at once.
\end{proof}

\begin{proof}[Proof of Theorem~\ref{prop:estimate-oblique-interior0}]
Throughout the proof $C$ will denote a generic positive constant 
only depending on $N$ and $\al$.
\par
From Proposition~\ref{prop:representation0}, 
the unique solution $v\in C^1_N(\ol{B})\cap C^2(B)$ of \eqref{oblique-laminar}-\eqref{equator} is given by 
\begin{equation*}
v(x)=W(x)+Z(x')=\int_0^{x_N} w(x',t) \, dt +Z(x'), \quad x=(x',x_N)\in B,
\end{equation*}
$w$ and $Z$ being the solutions of 
\eqref{Laplace-Dirichlet} and \eqref{eq-for-Z}, respectively.
We note that $Z$ is expressed as $Z=Z_1+Z_2$,
where $Z_1$ is the solution of \eqref{eq-for-Z} with $\pa_{x_N} w(x',0)$ replaced by $0$
and $Z_2$ is the solution of \eqref{eq-for-Z} with $\psi=0$.
Hence, it will be enough to prove the three estimates:
\begin{gather}
|W|_{1+\al,\ol{B}} +|\pa_{x_N} W|_{1+\al,\ol{B}}
 +|x_N D^2_{x'} W|_{\al,\ol{B}} \le C\,|\fhi|_{1+\al,\ol{B}},
\label{estimates-W}
\\
|Z_1|_{1+\al,\ol{D}} +|x_N D^2_{x'} Z_1|_{\al,\ol{B}} \le C\,|\psi|_{3/2+\al,\ol{D}}.
\label{estimates-Z1}
\\
|Z_2|_{1+\al,\ol{D}} +|x_N D^2_{x'} Z_2|_{\al,\ol{B}} \le C\,|\fhi|_{1+\al,\ol{B}}.
\label{estimates-Z2}
\end{gather}

We first derive these inequalities 
under the additional assumptions $\fhi \in C^{2+\al}(\ol{B})$ and $\psi \in C^{2+\al}(\ol{D})$.
We then have
\begin{equation}\label{regularity-wWZ}
w \in C^{2+\al}(\ol{B}), 
\quad 
W \in C^{2+\al}(\ol{B}), 
\quad 
Z_1 \in C^{2+\al}(\ol{D}).
\end{equation}
\par
We note that the following inequality holds:
\begin{equation}\label{estimate-Poisson-integral}
|w|_{1+\al,\ol{B}} \le C\,|\fhi|_{1+\al,\ol{B}}.
\end{equation}
Indeed, this is shown as follows.
Since 
\begin{equation*}
\int_{\cS} P_B(x;y) \, dS_y=1,
\quad
\int_{\cS} P_B(x;y) \, y \, dS_y =x,
\end{equation*}
we have
\begin{equation*}
\nabla w(x) =\nabla \fhi(x)+ \int_{\cS} \nabla_x P_B(x;y) \, [\fhi (y)-\fhi (x)-\nabla \fhi (x) \cdot (y-x)] \, dS_y.
\end{equation*}
This with Lemma~\ref{prop:derivatives-P} shows that
\begin{equation*}
|\nabla w(x)| \le |\nabla\fhi(x)|+C[\nabla\fhi]_{\al,\ol{B}} \int_{\cS} |x-y|^{-N+1+\al} dS_y.
\end{equation*}
We easily find that the integral on the right is finite and is bounded by some constant independent of $x$,
and hence $|\nabla w|_{0,\ol{B}} \le C|\nabla \fhi|_{\al,\ol{B}}$.
Since Proposition~\ref{lem:estimate-derivative-w} gives the inequality
$|D^2 w(x)| \le C|\fhi|_{1+\al,\ol{B}}(1-|x|^2)^{-1+\al}$,
we have that $[\nabla w]_{\al,\ol{B}} \le C|\fhi|_{1+\al,\ol{B}}$
by Lemma~\ref{lem:estimate-holder-seminorm}.
Therefore \eqref{estimate-Poisson-integral} holds.

First, we observe that \eqref{estimates-Z2}
easily follows from the Schauder estimates for the Poisson equation
and \eqref{estimate-Poisson-integral}.
In fact, we have that
\begin{equation*}
|Z_2|_{2+\al,\ol{D}} \le C\,|\pa_{x_N} w(\cdot,0)|_{\al,\ol{D}}
 \le C\,|w|_{1+\al,\ol{B}} \le C\,|\fhi|_{1+\al,\ol{B}}.
\end{equation*}

Next, we prove \eqref{estimates-W}.
It is clear that
\begin{equation*}
|W|_{1+\al,\ol{B}} \le C|w|_{1+\al,\ol{B}}.
\end{equation*}
This together with \eqref{estimate-Poisson-integral}
and the fact that $\pa_{x_N} W=w$ then yields:
\begin{equation*}
|W|_{1+\al,\ol{B}} +|\pa_{x_N} W|_{1+\al,\ol{\Om}}
 \le C\,|w|_{1+\al,\ol{B}} \le C\,|\fhi|_{1+\al,\ol{B}}.
\end{equation*}
Therefore, we only need to verify that
\begin{equation}
\label{estimate-W}
|x_N D^2_{x'} W|_{\al,\ol{B}} \le C|\fhi|_{1+\al,\ol{B}}.
\end{equation}
\par
To prove this inequality,
we examine pointwise estimates of $D^2_{x'} W$.
Proposition~\ref{lem:estimate-derivative-w} yields that
\begin{equation*}
|D^2_{x'} W(x)| =\left| \int_0^{x_N} D^2_{x'} w(x',t)dt \right|
 \le 
C\,|\fhi|_{1+\al,\ol{B}} \int_0^{|x_N|} \frac{dt}{(1-|x'|^2-t^2)^{1-\al}}.
\end{equation*}
We estimate the last integral in two ways. First, by monotonicity in $t$ and $|x_N|$, we see that the integral can be bounded by $(1-|x|^2)^{-1+\al}$. From this, we infer that
\begin{equation}
\label{D2W-estimate1}
|D^2_{x'} W(x)| \le C\,|\fhi|_{1+\al,\ol{B}}\, (1-|x|^2)^{-1+\al}.
\end{equation}
Secondly, by the inequality 
\begin{equation}\label{inequality-xN}
x_N^2 \le 1-|x'|^2 \ \mbox{ for } \ (x',x_N)\in B, 
\end{equation}
we get that
$$
\int_0^{|x_N|} \frac{dt}{(1-|x'|^2-t^2)^{1-\al}} \le
\int_0^{|x_N|} \frac{dt}{(x_N^2 -t^2)^{1-\al}}=|x_N|^{-1+2\al} \int_0^1 \frac{ds}{(1-s^2)^{1-\al}},
$$
after the change of variable $t=|x_N|\, s$.
This gives the bound:
\begin{equation}
\label{D2W-estimates2}
|x_N D^2_{x'} W|_{0,\ol{B}} \le C|\fhi|_{1+\al,\ol{B}}.
\end{equation}
\par
In order to estimate the H\"{o}lder seminorm $[x_N D^2_{x'} W]_{\al,\ol{B}}$, 
we consider the derivatives of $D^2_{x'} W$.
We use Proposition~\ref{lem:estimate-derivative-w} to obtain
\begin{equation*}
\label{D3W-estimate0}
|D^3_{x'} W(x)|=\left|\int_0^{x_N} D^3_{x'} w(x',t) dt\right|
\le C\,|\fhi|_{1+\al,\ol{B}} \int_0^{|x_N|} \frac{dt}{(1-|x'|^2-t^2)^{2-\al}}.
\end{equation*}
Applying Lemma~\ref{lem:estimate-integral0} to the right-hand side,
we infer that
\begin{equation}\label{D3W-estimate1}
|x_N D^3_{x'} W(x)| \le C\,|\fhi|_{1+\al,\ol{B}}\, (1-|x|^2)^{-1+\al}.
\end{equation}
Furthermore, we see from Proposition~\ref{lem:estimate-derivative-w} that
\begin{equation}\label{D3W-estimate2}
|\partial_{x_N} D^2_{x'} W(x)|=|D^2_{x'} w(x)| 
\le C\,|\fhi|_{1+\al,\ol{B}}\, (1-|x|^2)^{-1+\al}.
\end{equation}
Combining \eqref{D2W-estimate1}, \eqref{D3W-estimate1} and \eqref{D3W-estimate2}, 
we deduce that
\begin{equation*}
|x_N D^3_{x'} W(x)| +|\partial_{x_N} (x_N D^2_{x'} W(x))| 
\le C\,|\fhi|_{1+\al,\ol{B}}\, (1-|x|^2)^{-1+\al}.
\end{equation*}
Thus, by \eqref{regularity-wWZ} and Lemma~\ref{lem:estimate-holder-seminorm},
we obtain that $[x_N D^2_{x'} W]_{\al,\ol{B}} \le C\,|\fhi|_{1+\al,\ol{B}}$.
This together with \eqref{D2W-estimates2} shows that \eqref{estimate-W} holds.
\par
Finally, we verify \eqref{estimates-Z1}.
Note that the same estimates as in Proposition~\ref{lem:estimate-derivative-w}
and \eqref{estimate-Poisson-integral} hold
if $w$, $\fhi$ and $B$ are replaced by $Z_1$, $\psi$ and $D$, respectively.
By \eqref{estimate-Poisson-integral}, we see that
\begin{equation}\label{estimate0-Z1}
|Z_1|_{1+\al,\ol{D}} \le C\,|\psi|_{1+\al,\ol{D}} \le C\,|\psi|_{3/2+\al,\ol{D}}.
\end{equation}
\par
Proposition~\ref{lem:estimate-derivative-w} and \eqref{inequality-xN} show that
\begin{gather*}
|x_N D_{x'}^2 Z_1(x')| \le C\,|\psi|_{3/2,\ol{D}}\,|x_N|\,(1-|x'|)^{-\frac{1}{2}}
 \le C\,|\psi|_{3/2+\al,\ol{D}},
\\
|x_N D_{x'}^3 Z_1(x')| \le C\,|\psi|_{3/2+\al,\ol{D}} \, |x_N| \, (1-|x'|)^{-\frac{3}{2}+\al}
 \le C\,|\psi|_{3/2+\al,\ol{D}} \, (1-|x|)^{-1+\al},
\end{gather*}
and
\begin{multline*}
|\pa_{x_N}(x_N D_{x'}^2 Z_1(x'))| =|D_{x'}^2 Z_1(x')| \le
\\
C\,|\psi|_{1+\al,\ol{D}} \, (1-|x'|)^{-1+\al}
 \le C\,|\psi|_{3/2+\al,\ol{D}} \, (1-|x|)^{-1+\al}.
\end{multline*}
Hence it follows from \eqref{regularity-wWZ} and Lemma~\ref{lem:estimate-holder-seminorm} that
\begin{equation*}
|x_N D_{x'}^2 Z_1|_{\al,\ol{B}} \le C\,|\psi|_{3/2+\al,\ol{D}}.
\end{equation*}
This and \eqref{estimate0-Z1} gives \eqref{estimates-Z1}.
\par
Now, if $\fhi\in C^{1+\al}(\ol{B})$ and $\psi\in C^{3/2+\alpha}(\overline{D})$, we take sequences 
$\{\fhi_j\} \subset C^{2+\al}(\ol{B})$ and $\{\psi_j\} \subset C^{2+\al}(\ol{D})$ such that 
\begin{gather}
\fhi_j \to \fhi \ \mbox{ in } \ C(\ol{B}),
\quad \psi_j \to \psi \ \mbox{ in } \ C(\ol{D}) \ \mbox{ as } \ j \to \infty,
\label{approximation-f1}
\\
|\fhi_j|_{1+\al,\ol{B}} \le C\,|\fhi|_{1+\al,\ol{B}},  
\quad
|\psi_j|_{3/2+\al,\ol{D}} \le C\,|\psi|_{3/2+\al,\ol{D}},
\quad j=1, 2, \ldots.
\label{approximation-f2}
\end{gather}
Let $w_j$ be a solution of \eqref{Laplace-Dirichlet} with $\fhi=\fhi_j$. 
Then, \eqref{approximation-f1} and the Schauder interior estimates for Poisson's equation
give that $w_j \to w$ in $C^2_{\rm loc}(B)$ as $j \to \infty$.
Hence $W_j(x)=\int_0^{x_N} w_j(x',t)\,dt$ converges to $W$ in $C^2_{\rm loc}(B)$.
Since the inequality 
$|x_N D^2_{x'} W_j|_{\al,\ol{B}} \le C|\fhi_j|_{1+\al,\ol{B}}$ is valid,
we obtain \eqref{estimate-W} by using \eqref{approximation-f2} and letting $j \to \infty$.
In a similar way,
approximating $\psi$ by $\psi_j$ gives that 
\eqref{estimates-Z1} is still valid for $\psi \in C^{3/2+\al}(\ol{D})$.
Thus the proof is complete.
\end{proof}

\section{Local existence for the \\ interior Backus problem
with symmetric data}
\label{sec:backus}

This section is devoted to the proof of Theorem~\ref{thm:existence0}.

\subsection{The nonlinear operator $\cT$}

We define an operator $\cT$ by setting
\begin{equation*}
\cT[\fhi]=|\na v|^2,
\end{equation*}
where $v$ satisfies \eqref{oblique-laminar}-\eqref{equator} with $\psi=0$. The next proposition shows that $\cT$ is locally bounded and locally Lipschitz continuous on bounded subsets of $C_{\rm even}^{1+\al}(\ol{B})$.

\begin{prop}\label{lem:T-estimates}
We have that $\cT$ is a mapping from $C_{\rm even}^{1+\al}(\ol{B})$ into itself.
Furthermore, there are positive constants $C_1$ and $C_2$ such that
\begin{gather}
\bigl| \cT[\fhi]\bigr|_{1+\al,\ol{B}} \le C_1\,|\fhi|_{1+\al,\ol{B}}^2,
\label{T-estimate01}
\\
\bigl| \cT[\fhi_1] -\cT[\fhi_2]\bigr|_{1+\al,\ol{B}}
 \le C_2\, \left( |\fhi_1|_{1+\al,\ol{B}} +|\fhi_2|_{1+\al,\ol{B}}\right) |\fhi_1-\fhi_2|_{1+\al,\ol{B}}
\label{T-estimate02}
\end{gather}
for all $\fhi, \fhi_1, \fhi_2 \in C_{\rm even}^{1+\al}(\ol{B})$.
\end{prop}

To prove this proposition,
we need the following simple lemma.
\begin{lem}\label{lem:om-estimate0}
Let $k$ be a non-negative integer and take $\al \in (0,1)$.
Suppose that a function $v$ defined on $\ol{B}$ 
is such that $\pa_{x_N} v$ in $C^{k+\al}(\ol{B})$ 
and is zero on $\ol{D}\times\{ 0\}$.
Then the function defined by
\begin{equation}
\label{om-definition}
\om(x)=\begin{cases}
\displaystyle \frac{v(x)}{x_N} & \mbox{ for } \ x_N \neq 0,
\\
\pa_{x_N} v(x',0) & \mbox{ for } \ x_N=0,
\end{cases}
\end{equation}
belongs to $C^{k+\al}(\ol{B})$ and satisfies the inequality:
\begin{equation}\label{om-estimate0}
|\om|_{k+\al,\ol{B}} \le |\pa_{x_N} v|_{k+\al,\ol{B}}.
\end{equation}
\end{lem}
\begin{proof}
By the fundamental theorem of calculus, we have that
$$
v(x',x_N)=x_N \int_0^1 \pa_{x_N}v(x', x_N t)\,dt,
$$
and hence $\om$ can be written as
\begin{equation*}
\om(x)=\int_0^1 \pa_{x_N} v(x',x_N t)\, dt.
\end{equation*}
Thus, we see that the partial derivative $D_x^\be \om$ exists 
for any multi-index $\be=(\be_1,\ldots,\be_N)$ with $|\be| \le k$
and is given by
\begin{equation*}
D^\be_x \om(x)=\int_0^1 t^{\be_N} D^\be_x \pa_{x_N} v(x',x_N t) dt.
\end{equation*}
The assertion then easily follows from this formula.
\end{proof}

\begin{proof}[Proof of Proposition~\ref{lem:T-estimates}]
Throughout the proof,
$i$ is any index in $\{1,2,\ldots,N-1\}$
and $C$ is a generic positive constant only depending on $N$ and $\al$.
\par
(i) Let $\fhi \in C_{\rm even}^{1+\al}(\ol{B})$ and
let $v$ denote the solution of the problem \eqref{oblique-laminar}--\eqref{equator} with $\psi=0$.
Since $\fhi$ is even in $x_N$,
we see that the function $-v(x',-x_N)$ 
is also a solution of \eqref{oblique-laminar}--\eqref{equator} with $\psi=0$. By uniqueness, 
we infer that $v$ is odd in $x_N$.
In particular, $v(x',0)=0$,
and hence
\begin{equation}\label{T-derivative}
\pa_{x_j} (\pa_{x_i} v)^2 =2\pa_{x_i} v \, \pa^2_{x_ix_j} v
 =2\left( \pa_{x_i} \om \right) \left( x_N \pa^2_{x_ix_j} v \right)
\ \mbox{ for } \ j=1,\ldots,N,
\end{equation}
where $\om$ is a function given by \eqref{om-definition}. 
From Theorem~\ref{prop:estimate-oblique-interior0} and Lemma~\ref{lem:om-estimate0},
we see that the right-hand side of this equality is in $C^{\al} (\ol{B})$.
Therefore, we have that $(\pa_{x_i} v)^2 \in C^{1+\al} (\ol{B})$.
This together with the fact that $(\pa_{x_N} v)^2 \in C^{1+\al} (\ol{B})$,
which follows from Theorem~\ref{prop:estimate-oblique-interior0}, gives:
\begin{equation*}
\cT[\fhi]=|\nabla v|^2 \in C^{1+\al}(\ol{B}).
\end{equation*}
Since the fact that $v$ is odd in $x_N$ yields that $|\na v(x',x_N)|=|\na v(x',-x_N)|$,
we conclude that $\cT$ is a mapping 
from $C_{\rm even}^{1+\al}(\ol{B})$ to itself.
\par
(ii) Let us derive \eqref{T-estimate01}.
By \eqref{T-derivative}, we have:
\begin{multline*}
\left| (\pa_{x_i} v)^2\right|_{1+\al,\ol{B}}
 =\left| (\pa_{x_i} v)^2\right|_{0,\ol{B}}
  +\sum_{j=1}^N \left| \pa_{x_j} (\pa_{x_i} v)^2\right|_{\al,\ol{B}}=
\\
\left| (\pa_{x_i} v)^2\right|_{0,\ol{B}}
 +2\sum_{j=1}^{N-1} \left| \left( \pa_{x_i} \om \right)
  \left( x_N \pa^2_{x_ix_j} v \right) \right|_{\al,\ol{B}}
   +2\left| \pa_{x_i} v \,\pa^2_{x_ix_N} v) \right|_{\al,\ol{B}}.
\end{multline*}
From \eqref{estimate-oblique-interior}, 
the first and third terms of the rightest-hand side are handled as
\begin{gather*}
\left| (\pa_{x_i} v)^2\right|_{0,\ol{B}}
 \le \left| \pa_{x_i} v \right|_{0,\ol{B}}^2
  \le |v|_{1+\al,\ol{B}}^2 \le C\,|\fhi|_{1+\al,\ol{B}}^2,
\\
\left| \pa_{x_i} v \, \pa^2_{x_ix_N} v \right|_{\al,\ol{B}}
 \le \left| \pa_{x_i} v \right|_{\al,\ol{B}} \left| \pa^2_{x_ix_N} v \right|_{\al,\ol{B}}
  \le |v|_{1+\al,\ol{B}} \left| \pa_{x_N} v\right|_{1+\al,\ol{B}}
   \le C|\fhi|_{1+\al,\ol{B}}^2.
\end{gather*}
Furthermore, \eqref{estimate-oblique-interior} and \eqref{om-estimate0} 
show that 
\begin{multline*}
\left| \left( \pa_{x_i} \om \right) \left( x_N \pa^2_{x_ix_j} v \right) \right|_{\al,\ol{B}}
 \le \left| \pa_{x_i} \om \right|_{\al,\ol{B}}
  \left| x_N \pa^2_{x_ix_j} v\right|_{\al,\ol{B}} \le \\
\left| \om \right|_{1+\al,\ol{B}}
    \left| x_N \pa^2_{x_ix_j} v\right|_{\al,\ol{B}}
\le \left| \pa_{x_N} v \right|_{1+\al,\ol{B}}
 \left| x_N \pa^2_{x_ix_j} v\right|_{\al,\ol{B}}
  \le C\,|\fhi|_{1+\al,\ol{B}}^2,
\end{multline*}
for $j=1,\ldots,N-1$.
From these estimates it follows that
\begin{equation}
\label{T-estimate01-1}
\left| (\pa_{x_i} v)^2\right|_{1+\al,\ol{B}} \le C\,|\fhi|_{1+\al,\ol{B}}^2.
\end{equation}
In order to estimate  $(\pa_{x_N} v)^2$,
we use \eqref{estimate-oblique-interior} to find that
\begin{equation}
\label{T-estimate01-2}
\left| (\pa_{x_N} v)^2\right|_{1+\al,\ol{B}}
 \le C\,\left| \pa_{x_N} v\right|_{1+\al,\ol{B}}^2 \le C\,|\fhi|_{1+\al,\ol{B}}^2.
\end{equation}
The combination of  \eqref{T-estimate01-1} and \eqref{T-estimate01-2} then gives
\eqref{T-estimate01}.
\par
(iii)
It remains to prove \eqref{T-estimate02}.
For $m=1,2$,
let $v_m$ be the solution of the problem \eqref{oblique-laminar}-\eqref{equator}
with $\fhi=\fhi_m \in C_{\rm even}^{1+\al}(\ol{B})$ and $\psi=0$. Also, let $\om_m$ be defined by \eqref{om-definition} with $v=v_m$. 
\par
It is clear that
\begin{multline*}
\left| (\pa_{x_i} v_1)^2 -(\pa_{x_i} v_2)^2\right|_{0,\ol{B}}\le \\
\left( |v_1|_{0,\ol{B}} +|v_2|_{0,\ol{B}}\right)
 |v_1 -v_2|_{0,\ol{B}} \le
\left( |v_1|_{1+\al,\ol{B}} +|v_2|_{1+\al,\ol{B}}\right)
 |v_1 -v_2|_{1+\al,\ol{B}}.
\end{multline*}
Hence, \eqref{estimate-oblique-interior} easily gives that
\begin{equation}
\label{T-estimate02-1}
\left| (\pa_{x_i} v_1)^2 -(\pa_{x_i} v_2)^2\right|_{0,\ol{B}}
\le C\,\left( |\fhi_1|_{1+\al,\ol{B}} +|\fhi_2|_{1+\al,\ol{B}}\right)
 |\fhi_1-\fhi_2|_{1+\al,\ol{B}}.
\end{equation}
Next, we write the differential identity:
\begin{multline*}
\sum_{j=1}^N \pa_{x_j} \left[ (\pa_{x_i} v_1)^2 -(\pa_{x_i} v_2)^2 \right]=2\,\sum_{j=1}^N \bigl[\pa_{x_i} v_1\, \pa^2_{x_ix_j} v_1
  -\pa_{x_i} v_2\, \pa^2_{x_ix_j} v_2\bigr]= \\
2\,\sum_{j=1}^{N-1} \bigl[ (\pa_{x_i} \om_1) (x_N \pa^2_{x_ix_j} v_1 -x_N \pa^2_{x_ix_j} v_2)
 +x_N\,(\pa^2_{x_ix_j} v_2) (\pa_{x_i} \om_1 -\pa_{x_i} \om_2)\bigr]+ \\
2\,\bigl[(\pa_{x_i} v_1) (\pa^2_{x_ix_N} v_1 -\pa^2_{x_ix_N} v_2)
 +(\pa^2_{x_ix_N} v_2) (\pa_{x_i} v_1 -\pa_{x_i} v_2)\bigr].
\end{multline*}
We then take care of the last summand:
\begin{multline*}
\left| (\pa_{x_i} v_1) (\pa^2_{x_ix_N} v_1 -\pa^2_{x_ix_N} v_2)
 +(\pa^2_{x_ix_N} v_2) (\pa_{x_i} v_1 -\pa_{x_i} v_2) \right|_{\al,\ol{B}} \le
\\
\left| \pa_{x_i} v_1 \right|_{\al,\ol{B}}
 \left| \pa^2_{x_ix_N} v_1 -\pa^2_{x_ix_N} v_2 \right|_{\al,\ol{B}}
  +\left| \pa^2_{x_ix_N} v_2 \right|_{\al,\ol{B}}
   \left| \pa_{x_i} v_1 -\pa_{x_i} v_2 \right|_{\al,\ol{B}}
\le \\
\left| v_1 \right|_{1+\al,\ol{B}}
 \left| \pa_{x_N} v_1 -\pa_{x_N} v_2 \right|_{1+\al,\ol{B}}
  +\left| \pa_{x_N} v_2\right|_{1+\al,\ol{B}} \left| v_1 -v_2\right|_{1+\al,\ol{B}}.
\end{multline*}
Thus, we have that
\begin{multline*}
\left| (\pa_{x_i} v_1) (\pa^2_{x_ix_N} v_1 -\pa^2_{x_ix_N} v_2)
 +(\pa^2_{x_ix_N} v_2) (\pa_{x_i} v_1 -\pa_{x_i} v_2) \right|_{\al,\ol{B}} \le \\
C\,\left( |\fhi_1|_{1+\al,\ol{B}} +|\fhi_2|_{1+\al,\ol{B}}\right)
 |\fhi_1-\fhi_2|_{1+\al,\ol{B}}.
\end{multline*}
Moreover, for $j=1,\ldots,N-1$, we have:
\begin{multline*}
\left|  (\pa_{x_i} \om_1) (x_N \pa^2_{x_ix_j} v_1 -x_N \pa^2_{x_ix_j} v_2)
 +x_N\,(\pa^2_{x_ix_j} v_2) (\pa_{x_i} \om_1 -\pa_{x_i} \om_2) \right|_{\al,\ol{B}}
\le \\
\left| \om_1 \right|_{1+\al,\ol{B}}
 \left| x_N \pa^2_{x_ix_j} v_1 -x_N \pa^2_{x_ix_j} v_2 \right|_{\al,\ol{B}}
  +\left| x_N \pa^2_{x_ix_j} v_2 \right|_{\al,\ol{B}}
   \left| \om_1 -\om_2 \right|_{1+\al,\ol{B}} \le
\\
\left| \pa_{x_N} v_1 \right|_{1+\al,\ol{B}}
 \left| x_N \pa^2_{x_ix_j} v_1 -x_N \pa^2_{x_ix_j} v_2 \right|_{\al,\ol{B}}+
\left| x_N \pa^2_{x_ix_j} v_2 \right|_{\al,\ol{B}}
 \left| \pa_{x_N} v_1 -\pa_{x_N} v_2 \right|_{1+\al,\ol{B}}.
\end{multline*}
where we have used  \eqref{om-estimate0}.
Hence, \eqref{estimate-oblique-interior} gives:
\begin{multline}
\left|  (\pa_{x_i} \om_1) (x_N \pa^2_{x_ix_j} v_1 -x_N \pa^2_{x_ix_j} v_2)
 +x_N\,(\pa^2_{x_ix_j} v_2) (\pa_{x_i} \om_1 -\pa_{x_i} \om_2) \right|_{\al,\ol{B}} \le \\
C\,\left( |\fhi_1|_{1+\al,\ol{B}} +|\fhi_2|_{1+\al,\ol{B}}\right)
 |\fhi_1-\fhi_2|_{1+\al,\ol{B}}.
\label{T-estimate02-3}
\end{multline}

All in all, by \eqref{T-estimate02-1}--\eqref{T-estimate02-3}  
we deduce that
\begin{equation*}
\left| (\pa_{x_i} v_1)^2 -(\pa_{x_i} v_2)^2\right|_{1+\al,\ol{B}}
 \le C\,\left( |\fhi_1|_{1+\al,\ol{B}} +|\fhi_2|_{1+\al,\ol{B}}\right) |\fhi_1-\fhi_2|_{1+\al,\ol{B}},
\end{equation*}
for $i=1,\dots, N-1$.
\par
On the other hand, we can also use \eqref{estimate-oblique-interior} to infer that
\begin{multline*}
\left| (\pa_{x_N} v_1)^2 -(\pa_{x_N} v_2)^2\right|_{1+\al,\ol{B}}\le \\
C\,\left( |\pa_{x_N} v_1|_{1+\al,\ol{B}}
 +|\pa_{x_N} v_2|_{1+\al,\ol{B}}\right) |\pa_{x_N} v_1 -\pa_{x_N} v_2|_{1+\al,\ol{B}} \le \\
C\,\left( |\fhi_1|_{1+\al,\ol{B}} +|\fhi_2|_{1+\al,\ol{B}}\right) |\fhi_1-\fhi_2|_{1+\al,\ol{B}}.
\end{multline*}
\par
In conclusion, we obtain \eqref{T-estimate02}, and the lemma follows.
\end{proof}

\subsection{The nonlinear operator $\tilde \cT_\psi$}

We fix a cut-off function $\eta \in C^\infty(\RR)$ satisfying
\begin{equation*}
\eta (t)=1 \ \mbox{ if } \ |t| \le \frac{1}{3},
\qquad
\eta (t)=0 \ \mbox{ if } \ |t| \ge \frac{2}{3}.
\end{equation*}
For a function $\phi$ defined on $\ol{B}$,
we set
\begin{equation*}
\label{definition-J}
{\cJ}[\phi](x)=\eta(x_N) \phi \left( \sqrt{1-x_N^2}e_1',x_N\right) +(1-\eta(x_N)) \phi(x),
\quad x=(x',x_N) \in \ol{B}.
\end{equation*}
Here $e_1'=(1,0,\ldots,0) \in \RR^{N-1}$.
Then, for fixed $\psi \in C^{3/2+\al}(\ol{D})$, 
we define an operator $\tilde \cT_\psi$ by
\begin{equation*}
\tilde \cT_\psi [\fhi]=\cJ \left[ |\na v|^2\right],
\end{equation*}
where $v$ is the solution of \eqref{oblique-laminar}-\eqref{equator}.
Our goal here is to obtain the properties of $\tilde \cT_\psi$,
which enables us to solve problem \eqref{backus-problem} in $C^{1+\al}_{\rm ax}(\ol{B})$.

\begin{prop}\label{lem:tT-estimates}
Let $\psi \in C^{3/2+\al}(\ol{D})$. 
Then the following hold.
\begin{itemize}
\item[(i)]
The operator $\tilde \cT_\psi$ is a mapping from $C^{1+\al}(\ol{B})$ into itself.
Moreover, there exist positive constants $C_3$ and $C_4$ such that
\begin{equation}
\bigl| \tilde \cT_\psi[\fhi]\bigr|_{1+\al,\ol{B}}
 \le C_3 \, \left( |\fhi|_{1+\al,\ol{B}}^2 +|\psi|_{3/2+\al,\ol{D}}^2\right),
\label{T-estimate03}
\end{equation}
\begin{multline}
\bigl| \tilde \cT_\psi[\fhi_1] -\tilde \cT_\psi[\fhi_2]\bigr|_{1+\al,\ol{B}}
 \le \\
C_4\, \left( |\fhi_1|_{1+\al,\ol{B}} +|\fhi_2|_{1+\al,\ol{B}} +|\psi|_{3/2+\al,\ol{D}} \right)
  |\fhi_1-\fhi_2|_{1+\al,\ol{B}}
\label{T-estimate04}
\end{multline}
for all $\fhi, \fhi_1, \fhi_2 \in C^{1+\al}(\ol{B})$.
\item[(ii)]
If $\fhi \in C^{1+\al}_{\rm ax}(\ol{B})$ and $\psi$ is constant, then
\begin{equation}\label{tT-symmetry}
\tilde \cT_\psi [\fhi] \in C^{1+\al}_{\rm ax}(\ol{B})
\ \mbox{ and } \ \tilde \cT_\psi [\fhi] =|\na v|^2 \ \mbox{ on } \ \cS.
\end{equation}
Here $v$ is the solution of \eqref{oblique-laminar}-\eqref{equator}.
\end{itemize}
\end{prop}

Before proving the proposition,
we examine properties of $\cJ$.
\begin{lem}\label{lem:J-S}
If a function $\phi$ defined on $\ol{B}$ satisfies $\phi(x',x_N)=\phi(|x'|e_1',x_N)$, 
then
\begin{equation*}
{\cJ}[\phi]=\phi \ \mbox{ on } \ \cS.
\end{equation*}
\end{lem}
\begin{proof}
Let $x=(x',x_N) \in \cS$.
Then, since $|x'|^2+x_N^2=1$,
we have that
\begin{equation*}
\phi \left( \sqrt{1-x_N^2}e_1',x_N\right) =\phi (|x'|e_1',x_N) =\phi(x).
\end{equation*}
This gives that
${\cJ}[\phi](x)=\eta(x_N) \phi (x) +(1-\eta(x_N)) \phi(x) =\phi(x)$,
as desired.
\end{proof}

\begin{lem}\label{lem:J-estimate}
Suppose that $\phi \in C^\al(\ol{B}) \cap C^1(B)$ satisfies 
$\partial_{x_N} \phi \in C^\al(\ol{B})$ and $x_N \nabla_{x'} \phi \in C^\al(\ol{B})$.
Then ${\cJ}[\phi] \in C^{1+\al}(\ol{B})$ and
\begin{equation*}
|{\cJ}[\phi]|_{1+\al,\ol{B}}
 \le C\left( |\phi|_{\al,\ol{B}} +|\partial_{x_N} \phi|_{\al,\ol{B}}
  +|x_N \na_{x'} \phi|_{\al,\ol{B}} \right),
\end{equation*}
where $C$ is a positive constant independent of $\phi$.
\end{lem}

\begin{proof}
Let 
\begin{equation*}
B_1=\left\{ (x',x_N) \in B : |x_N|<\frac{2}{3} \right\},
\qquad
B_2=\left\{ (x',x_N) \in B : |x_N|>\frac{1}{3} \right\},
\end{equation*}
and define the mapping
\begin{equation*}
\xi :\ol{B} \ni (x',x_N) \mapsto \xi(x)=\left( \sqrt{1-x_N^2}e_1',x_N\right) \in \ol{B}.
\end{equation*}
We first show that 
\begin{gather}
\phi \circ \xi \in C^{1+\al}(\ol{B_1}), \qquad \phi \in C^{1+\al}(\ol{B_2}),
\label{J-estimate0}
\\
|\phi \circ \xi|_{1+\al,\ol{B_1}} +|\phi|_{1+\al,\ol{B_2}}
 \le C\left( |\phi|_{\al,\ol{B}} +|\partial_{x_N} \phi|_{\al,\ol{B}}
  +|x_N \na_{x'} \phi|_{\al,\ol{B}} \right).
\label{J-estimate1}
\end{gather}
Here and subsequently, 
$C$ denotes a positive constant independent of $\phi$.

To prove $\phi \circ \xi \in C^{1+\al}(\ol{B_1})$,
we observe that $\phi \circ \xi$ depends only on $x_N$ and
\begin{equation*}
\partial_{x_N} (\phi \circ \xi)
 =(\partial_{x_N} \phi) \circ \xi -\frac{1}{\sqrt{1-x_N^2}} (x_N \partial_{x_1} \phi) \circ \xi.
\end{equation*}
By assumption and the fact that $\xi$ and $1/\sqrt{1-x_N^2}$ are smooth on $\ol{B_1}$,
we see that the right-hand side of this equality is in $C^{\al}(\ol{B_1})$.
Moreover,
\begin{multline*}
|\partial_{x_N} (\phi \circ \xi)|_{\al,\ol{B_1}}
 \le \left| (\partial_{x_N} \phi) \circ \xi \right|_{\al,\ol{B_1}}
  +\left| \frac{1}{\sqrt{1-x_N^2}} (x_N \partial_{x_1} \phi) \circ \xi \right|_{\al,\ol{B_1}} \le
\\
C|\partial_{x_N} \phi |_{\al,\ol{B}}
 +C |(x_N \partial_{x_1} \phi) \circ \xi|_{\al,\ol{B_1}}
  \le C|\partial_{x_N} \phi |_{\al,\ol{B}} +C|x_N \partial_{x_1} \phi|_{\al,\ol{B}}.
\end{multline*}
Therefore, we have that $\phi \circ \xi \in C^{1+\al}(\ol{B_1})$ and
\begin{multline}
\label{J-estimate01}
|\phi \circ \xi|_{1+\al,\ol{B_1}} =|\phi  \circ \xi|_{0,\ol{B_1}}
 +|\partial_{x_N} (\phi \circ \xi)|_{\al,\ol{B}} \le \\
 |\phi|_{0,\ol{B}} +C|\partial_{x_N} \phi |_{\al,\ol{B}} +C|x_N \partial_{x_1} \phi|_{\al,\ol{B}}.
\end{multline}

Since $x_N^{-1}$ is smooth on $\ol{B_2}$,
we deduce that $\phi =x_N^{-1} \cdot (x_N \phi) \in C^{1+\al}(\ol{B_2})$ and
\begin{multline}\label{J-estimate02}
|\phi|_{1+\al,\ol{B_2}} =|\phi|_{0,\ol{B_2}} +|\partial_{x_N} \phi|_{\al,\ol{B_2}}
 +|x_N^{-1} \cdot x_N \nabla_{x'} \phi |_{\al,\ol{B_2}} \le
\\
|\phi|_{0,\ol{B}} +|\partial_{x_N} \phi|_{\al,\ol{B}}
 +C|x_N \nabla_{x'} \phi |_{\al,\ol{B}}.
\end{multline}
By \eqref{J-estimate01} and \eqref{J-estimate02}, 
we obtain \eqref{J-estimate1}.

We note that $\eta$ and $1-\eta$ vanish on 
$\ol{B} \setminus B_1$ and $\ol{B} \setminus B_2$, respectively.
This with \eqref{J-estimate0} shows that
$\cJ[\phi]=\eta (\phi \circ \xi) +(1-\eta)\phi \in C^{1+\al}(\ol{B})$.
Furthermore, 
\begin{equation*}
|{\cJ}[\phi]|_{1+\al,\ol{B}}
 \le C|\eta (\phi \circ \xi)|_{1+\al,\ol{B_1}} +C|(1-\eta) \phi|_{1+\al,\ol{B_2}}
  \le C|\phi \circ \xi|_{1+\al,\ol{B_1}} +C|\phi|_{1+\al,\ol{B_2}}.
\end{equation*}
Combining this and \eqref{J-estimate1} proves the lemma.
\end{proof}

\begin{proof}[Proof of Proposition~\ref{lem:tT-estimates}]
Let $\fhi, \fhi_1,\fhi_2\in C^{1+\al}(\ol{B})$ and $\psi \in C^{3/2+\al}(\ol{D})$.
In the proof,
$C$ stands for a generic positive constant only independent of these functions.

Let $v$ stand for the solution of \eqref{oblique-laminar}-\eqref{equator}.
Then Theorem~\ref{prop:estimate-oblique-interior0} shows that
the function $\phi=|\na v|^2$ satisfies 
$\phi \in C^\al(\ol{B}) \cap C^1(B)$, 
$\pa_{x_N} \phi =2\na v \cdot \na \partial_{x_N} v \in C^\al(\ol{B})$
and $x_N \pa_{x_j} \phi =2x_N \na v \cdot \na \partial_{x_j} v \in C^\al(\ol{B})$
for $j=1,\ldots,N-1$.
Hence, using Lemma~\ref{lem:J-estimate},
we see that $\tilde \cT_\psi[\fhi]={\cJ}[\phi] \in C^{1+\al}(\ol{B})$ and
\begin{equation*}
\bigl| \tilde \cT_\psi [\fhi] \bigr|_{1+\al,\ol{B}}
 \le C\left( \bigl| |\na v|^2 \bigl|_{\al,\ol{B}}
  +\bigl| \na v \cdot \na \partial_{x_N} v \bigl|_{\al,\ol{B}}
   +\sum_{j=1}^{N-1} \bigl| x_N \na v \cdot \na \partial_{x_j} v \bigl|_{\al,\ol{B}} \right).
\end{equation*}
Each term of the right-hand side is estimated as
\begin{equation*}
\bigl| |\na v|^2 \bigl|_{\al,\ol{B}} \le |v|_{1+\al,\ol{B}}^2,
\qquad
\bigl| \na v \cdot \na \partial_{x_N} v \bigl|_{\al,\ol{B}}
 \le C|v|_{1+\al,\ol{B}} |\partial_{x_N} v|_{1+\al,\ol{B}},
\end{equation*}
and
\begin{multline*}
\sum_{j=1}^{N-1} \bigl| x_N \na v \cdot \na \partial_{x_j} v \bigl|_{\al,\ol{B}}
 \le C\sum_{j=1}^{N-1}|v|_{1+\al,\ol{B}} |x_N \na \partial_{x_j} v|_{\al,\ol{B}} \le
\\ 
C|v|_{1+\al,\ol{B}} \left( |x_N D_{x'}^2 v|_{\al,\ol{B}} 
 +|\partial_{x_N} v|_{1+\al,\ol{B}} \right).
\end{multline*}
Therefore, by \eqref{estimate-oblique-interior}, 
we conclude that $\tilde \cT_\psi$ is a mapping from $C^{1+\al}(\ol{B})$ into itself
and that \eqref{T-estimate03} holds.
\par
Inequality \eqref{T-estimate04} is shown as follows.
For $i=1,2$, we denote by $v_i$ the solution of \eqref{oblique-laminar}-\eqref{equator}
with $\fhi=\fhi_i$ and $\psi=\psi_i$.
We see from Lemma~\ref{lem:J-estimate} that
\begin{multline}\label{T-estimate040}
\bigl| \tilde \cT_\psi [\fhi_1] -\tilde \cT_\psi [\fhi_2] \bigr|_{1+\al,\ol{B}}
 =\bigl| \cJ \bigl[ |\na v_1|^2 -|\na v_2|^2\bigr] \bigr|_{1+\al,\ol{B}} \le
\\
C\left( \bigl| |\na v_1|^2 -|\na v_2|^2 \bigr|_{0,\ol{B}}
 +\bigl| \partial_{x_N} \left( |\na v_1|^2 -|\na v_2|^2 \right) \bigr|_{\al,\ol{B}} + \right.
\\ \left. \sum_{j=1}^{N-1} \bigl| x_N \partial_{x_j} \left( |\na v_1|^2 -|\na v_2|^2 \right)
 \bigr|_{\al,\ol{B}} \right).
\end{multline}
For abbreviation, 
we write $w_1=v_1 +v_2$ and $w_2=v_1 -v_2$.
Then the first and second terms on the right of the above inequality can be handled as
\begin{equation}
\label{T-estimate041}
\bigl| |\na v_1|^2 -|\na v_2|^2 \bigr|_{0,\ol{B}}
 =\bigl| \na w_1 \cdot \na w_2 \bigr|_{0,\ol{B}}
  \le C|w_1|_{1+\al,\ol{B}} |w_2|_{1+\al,\ol{B}},
\end{equation}
and
\begin{multline*}
\bigl| \partial_{x_N} \left( |\na v_1|^2 -|\na v_2|^2 \right) \bigr|_{\al,\ol{B}}
 =\bigl| \na \partial_{x_N} w_1 \cdot \na w_2
  +\na w_1 \cdot \na \partial_{x_N} w_2 \bigr|_{\al,\ol{B}} \le
\\
C|\partial_{x_N} w_1|_{1+\al,\ol{B}} |w_2|_{1+\al,\ol{B}}
 +C|w_1|_{1+\al,\ol{B}} |\partial_{x_N} w_2|_{1+\al,\ol{B}}.
\end{multline*}
The other terms are estimated as
\begin{multline*}
\bigl| x_N \partial_{x_j} \left( |\na v_1|^2 -|\na v_2|^2 \right) \bigr|_{\al,\ol{B}}=
 \bigl| x_N \na \partial_{x_j} w_1 \cdot \na w_2
  +x_N \na w_1 \cdot \na \partial_{x_j} w_2 \bigr|_{\al,\ol{B}} \le
\\
|x_N \nabla \partial_{x_j} w_1|_{\al,\ol{B}} |w_2|_{1+\al,\ol{B}}
 +|w_1|_{1+\al,\ol{B}} |x_N \nabla \partial_{x_j} w_2|_{\al,\ol{B}},
\end{multline*}
and hence
\begin{multline*}
\sum_{j=1}^{N-1} \bigl| x_N \partial_{x_j} \left( |\na v_1|^2 -|\na v_2|^2 \right)
 \bigr|_{\al,\ol{B}} \le
\\
\left( |x_N D_{x'}^2 w_1|_{\al,\ol{B}} + 
|\partial_{x_N} w_1|_{1+\al,\ol{B}}\right) |w_2|_{1+\al,\ol{B}}+ \\
|w_1|_{1+\al,\ol{B}} \left( |x_N D_{x'}^2 w_2|_{\al,\ol{B}} +|\partial_{x_N} w_2|_{1+\al,\ol{B}} \right).
\end{multline*}
We note that $w_1$ satisfies \eqref{oblique-laminar}-\eqref{equator} 
with $\fhi$ and $\psi$ replaced by $\fhi_1+\fhi_2$ and $2\psi$, respectively.
Hence it follows from Theorem~\ref{prop:estimate-oblique-interior0} that
\begin{multline*}
|w_1|_{1+\al,\ol{B}} +|\partial_{x_N} w_1|_{1+\al,\ol{B}} +|x_N D_{x'}^2 w_1|_{\al,\ol{B}} \le \\ C\left( |\fhi_1|_{1+\al,\ol{B}} +|\fhi_2|_{1+\al,\ol{B}} +|\psi|_{3/2+\al,\ol{D}} \right).
\end{multline*}
Moreover, since $w_2$ solves \eqref{oblique-laminar}-\eqref{equator} 
with $\fhi=\fhi_1-\fhi_2$ and $\psi=0$,
Theorem~\ref{prop:estimate-oblique-interior0} shows that
\begin{equation}\label{T-estimate045}
|w_2|_{1+\al,\ol{B}} +|\partial_{x_N} w_2|_{1+\al,\ol{B}} +|x_N D_{x'}^2 w_2|_{\al,\ol{B}}
 \le C|\fhi_1-\fhi_2|_{1+\al,\ol{B}}.
\end{equation}
We thus obtain \eqref{T-estimate04}
by plugging \eqref{T-estimate041}-\eqref{T-estimate045} into \eqref{T-estimate040}.

It remains to prove (ii).
To this end, we assume that $\fhi \in C^{1+\al}_{\rm ax}(\ol{B})$ and that $\psi$ is constant.
Then, we can directly check that for any $(N-1) \times (N-1)$ orthogonal matrix ${\cR}$, 
the function $v({\cR}x',x_N)$ also satisfies \eqref{oblique-laminar}-\eqref{equator}.
Hence Proposition~\ref{prop:uniqueness} gives that $v(x',x_N)=v({\cR}x',x_N)$.
In particular, we have that $|\na v (x',x_N)|^2=|\na v(|x'|e_1',x_N)|^2$.
Therefore (i) of Proposition~\ref{lem:tT-estimates} and Lemma~\ref{lem:J-S} 
show that \eqref{tT-symmetry} holds,
and the proof is complete.
\end{proof}

\subsection{Contraction mappings and the proof of Theorem~\ref{thm:existence0}}

For $g \in C^{1+\al}(\ol{B})$ and $h \in \RR$,
we define operators $\Psi_g$ and $\tilde \Psi_{g,h}$ by
$$
\Psi_g[\fhi]=\frac{1}{2} \left( g^2-1 -\cT [\fhi]\right), 
\qquad
\tilde \Psi_{g,h}[\fhi]=\frac{1}{2} \left( g^2-1 -\tilde \cT_h [\fhi]\right).
$$
Note that, by Propositions~\ref{lem:T-estimates} and \ref{lem:tT-estimates}, 
$\Psi_g$ is a mapping from $C_{\rm even}^{1+\al}(\ol{B})$ into itself
and $\tilde \Psi_{g,h}$ is a mapping from $C_{\rm ax}^{1+\al}(\ol{B})$ into itself.
We also define closed sets of $C^{1+\al}(\ol{B})$ by
\begin{gather*}
X_g=\{ \fhi \in C_{\rm even}^{1+\al}(\ol{B}): |\fhi|_{1+\al,\ol{B}}\le |g^2-1|_{1+\al,\ol{B}} \},
\\
\tilde X_{g,h}=\{ \fhi \in C_{\rm ax}^{1+\al}(\ol{B}):
 |\fhi|_{1+\al,\ol{B}}\le |g^2-1|_{1+\al,\ol{B}} +|h|\},
\end{gather*}
and positive constants $\de_1$ and $\de_2$ by
\begin{equation*}
\de_1=\min \left\{ \frac{1}{C_1}, \frac{\la}{C_2} \right\},  
\qquad
\de_2=\min \left\{ \frac{1}{2C_3}, \frac{2\la}{3C_4} \right\}.
\end{equation*}
Here, $\la \in (0,1)$ is an arbitrary fixed constant,
and $C_1$, $C_2$, $C_3$, $C_4$ are the constants 
given in Propositions~\ref{lem:T-estimates} and \ref{lem:tT-estimates}.

\begin{lem}\label{lem:fixed-point}
The following hold.
\begin{itemize}
\item[(i)]
If $g \in C_{\rm even}^{1+\al}(\ol{B})$ satisfies $|g^2-1|_{1+\al,\ol{B}} \le \de_1$,
then $\Psi_g$ has a unique fixed point in $X_g$.
\item[(ii)]
If $g \in C_{\rm ax}^{1+\al}(\ol{B})$ and $h \in \RR$ satisfy 
$|g^2-1|_{1+\al,\ol{B}} +|h| \le \de_2$,
then $\tilde \Psi_{g.h}$ has a unique fixed point in $\tilde X_{g.h}$.
\end{itemize}
\end{lem}
\begin{proof}
We show (i).
Set $\de=|g^2-1|_{1+\al,\ol{B}}$, so that $\de \le \de_1$.
For $\fhi \in X_g$, we see from Proposition~\ref{lem:T-estimates} that 
$\Psi_g[\fhi] \in C_{\rm even}^{1+\al}(\ol{B})$ and
\begin{equation*}
\bigl|\Psi_g[\fhi]\bigr|_{1+\al,\ol{B}}  \le \frac{1}{2} \left( \de +\bigl|\cT [\fhi]\bigr|_{1+\al,\ol{B}} \right)
 \le \frac{1}{2} \left( \de +C_1 |\fhi|_{1+\al,\ol{B}}^2 \right)
  \le \frac{1}{2} \left( 1+C_1 \de_1 \right) \de \le \de.
\end{equation*}
Hence we have that $\Psi_g(X_g)\subseteq X_g$.
\par
Furthermore, inequality \eqref{T-estimate02} shows that for $\fhi_1, \fhi_2 \in X_g$,
\begin{multline*}
\bigl|\Psi_g[\fhi_1] -\Psi_g[\fhi_2]\bigr|_{1+\al,\ol{B}} 
 =\frac{1}{2}\, \bigl|\cT[\fhi_1] -\cT[\fhi_2]\bigr|_{1+\al,\ol{B}} \le
\\
\frac{1}{2}\, C_2\, \left( |\fhi_1|_{1+\al,\ol{B}} +|\fhi_2|_{1+\al,\ol{B}}\right)
 |\fhi_1-\fhi_2|_{1+\al,\ol{B}} \le
C_2\, \de_1 |\fhi_1-\fhi_2|_{1+\al,\ol{B}}.
\end{multline*}
Thus, we infer that 
$$
\bigl|\Psi_g[\fhi_1] -\Psi_g[\fhi_2]\bigr|_{1+\al,\ol{B}}  \le 
\la\,|\fhi_1-\fhi_2|_{1+\al,\ol{B}},
$$
i.e. we have proved that $\Psi_g$ is a contraction mapping in $X_g$.
The Banach fixed point theorem then gives the desired conclusion.

The assertion (ii) can be proved in the same way, 
by applying Proposition~\ref{lem:tT-estimates} instead of Proposition~\ref{lem:T-estimates}.
\end{proof}

We are now in a position to prove Theorem~\ref{thm:existence0}.
\begin{proof}[Proof of Theorem~\ref{thm:existence0}]
We first prove (i).
Note that 
\begin{multline}
\label{g-1-estimate}
|g^2-1|_{1+\al,\ol{\Om}} \le C_0 |g+1|_{1+\al,\ol{\Om}} |g-1|_{1+\al,\ol{\Om}}
 \le \\
C_0 \left( |g-1|_{1+\al,\ol{\Om}} +2\right) |g-1|_{1+\al,\ol{\Om}}
\end{multline}
for some constant $C_0>0$.
Hence, if 
\begin{equation*}
|g-1|_{1+\al,\ol{\Om}} \le \sqrt{1+\frac{\de_1}{C_0}}-1,
\end{equation*}
we have that $|g^2-1|_{1+\al,\ol{\Om}} \le \de_1$.
\par
Suppose that the above condition is satisfied.
Then,  $\Psi_g$ has a unique fixed point $\fhi\in X_g$, by Lemma~\ref{lem:fixed-point}.
Let $v$ be the solution of \eqref{oblique-laminar}-\eqref{equator} with $\psi=0$.
Then $u=f+v$ is harmonic in $B$ and
\begin{multline}\label{boundary-condition}
|\na u|^2=|\na f|^2 +2\na f \cdot \na v +|\na v|^2 
=1+2\pa_{x_N} v +\cT [\fhi]= \\1+2\fhi +(g^2-1-2\Psi_g[\fhi])=g^2 \ \mbox{ on } \ \cS.
\end{multline}
This shows that $u$ is a solution of \eqref{backus-problem}.
Moreover, Theorem~\ref{prop:estimate-oblique-interior0} and the fact that $\fhi \in X_g$ 
give that $u$ is in $C_{\rm odd}^{1+\al}(\ol{B})$ and
\begin{equation*}
|u-f|_{1+\al,\ol{B}} =|v|_{1+\al,\ol{B}} \le C\,|\fhi|_{1+\al,\ol{B}}
 \le C\, |g^2-1|_{1+\al,\ol{B}} \le C\,|g-1|_{1+\al,\ol{B}},
\end{equation*}
where $C>0$ is a constant.
The assertion (i) thus holds if we take $\de_0$ 
such that $\de_0 \le \sqrt{1+\de_1/C_0}-1$.

For (ii), we may assume $h=0$, by considering $\tilde{u}=u-h$ instead of $u$. 
Then, (ii) can be shown in a similar way.
We see from \eqref{g-1-estimate} that $|g^2-1|_{1+\al,\ol{\Om}} \le \de_2$ if
\begin{equation*}
|g-1|_{1+\al,\ol{\Om}} \le \sqrt{1+\frac{\de_2}{C_0}}-1.
\end{equation*}
Under this condition, Lemma~\ref{lem:fixed-point} shows that  
$\tilde \Psi_{g,0}$ has a fixed point $\tilde \fhi\in \tilde X_{g,0}$.
We put $\tilde u=f+\tilde v$ for the solution $\tilde v$ of \eqref{oblique-laminar}-\eqref{equator} 
with $\fhi=\tilde \fhi$ and $\psi=0$.
Then $\tilde u$ solves \eqref{backus-problem},
since the same computation as in \eqref{boundary-condition} is valid thanks to \eqref{tT-symmetry}.
The fact that $\tilde u \in C_{\rm ax}^{1+\al}(\ol{B})$ and the inequality
$|\tilde u-f|_{1+\al,\ol{B}} \le C\,|g-1|_{1+\al,\ol{B}}$ 
follow from Theorem~\ref{prop:estimate-oblique-interior0}.
We have thus shown (ii), and the proof is complete.
\end{proof}

\smallskip

\section*{Acknowledgements}

The authors would like to thank the anonymous referee for useful suggestions for improving the readability of the presentation. 
The first author was partially supported by the Grant-in-Aid for Early-Career Scientists 19K14574, Japan Society for the Promotion of Science.
The second author was partially supported by the Gruppo Nazionale per l'Analisi Matematica, la Probabilit\`a e le loro Applicazioni (GNAMPA) dell'Istituto Nazionale di Alta Matematica (INdAM). 
The third author was supported in part by the Grant-in-Aid for Scientific Research (C) 20K03673, Japan Society for the Promotion of Science. 
This research started while the second author was visiting the Department of Mathematics of Tokyo Institute of Technology. He wants to thank their kind hospitality.


\begin{thebibliography}{DDO}


\bibitem{Al} Sh.~A.~Alimov, \emph{On a problem with an oblique derivative}, (Russian)
Differentsial'nye Uravneniya 17 (1981), no. 10, 1738--1751, 1915. English translation: Differential Equations 17 (1981), no. 10, 1073--1083 (1982).

\bibitem{Ba1}
G.~E.~Backus, \emph{Application of a non-linear boundary value problem for the Laplace's equation to gravity and geomagnetic surveys}, Quart. J. Mech. Appl. Math. XXI (1968), 195--221.

%
%
\bibitem{CK} D.~Colton, R.~Kress, Integral equation methods in scattering theory. John Wiley \& Sons, Inc., New York, 1983. 

\bibitem{DDO}
G.~D\'iaz, J.~J.~D\'iaz, J.~Otero, \emph{Construction of the maximal solution of Backus' problem
in geodesy and geomagnetism}, Stud. Geophys. Geod. 55 (2011), 415-440.



\bibitem{GT}
D.~Gilbarg, N.~S.~Trudinger,
Elliptic partial differential equations of second order.
Reprint of the 1998 edition, Springer-Verlag, Berlin,  2001.

\bibitem{HHIKL} M.~H\"am\"al\"ainen, R.~Hari, R.~J.~Ilmoniemi, J.~Knuutila, O.~V.~Lounasmaa, \emph{Magnetoencephalography -- theory, instrumentation, and applications to noninvasinve studies of the working human brain}, Review of Modern Physics, 65 (1993), 413--497.

\bibitem{Ha} R.~S.~Hamilton, \emph{The inverse function theorem of Nash and Moser}, Bull. Amer. Math. Soc. 7 (1982), 65--222. 

\bibitem{Jo} M.~C.~Jorge, \emph{Local existence of the solution to a nonlinear inverse problem in gravitation}, Quart. Appl. Math. 45 (1987), 287--292.

\bibitem{JM}
M.~C.~Jorge, R.~Magnanini, \emph{Explicit calculation of the solution to Backus problem with
condition for uniqueness}, J. Math. Anal. Appl. 173 (1993), 515--522.

\bibitem{KMO}
T.~Kan, R.~Magnanini, M.~Onodera,
\emph{Backus problem in geophysics: a resolution near the dipole in fractional Sobolev spaces},
NoDEA Nonlinear Differential Equations Appl. 29 (2022), no.~3, Paper No. 21, 29 pp.
 
\bibitem{LL} E.~H.~Lieb, M.~Loss, Analysis. American Mathematical Society, Providence, RI, 1997.

\bibitem{Ma}
R.~Magnanini, \emph{A fully nonlinear boundary value problem for the Laplace equation in dimension two}, Appl. Anal. 39 (2-3), 185--192.

\bibitem{Mo1} J.~Moser, \emph{A new technique for the construction of solutions of nonlinear differential equations}, Proc. N. A. S. 47 (1961), 1824--1831.

\bibitem{Mo2} J.~Moser,  \emph{A rapidly convergent iteration method and non-linear partial differential equations -- I}, Ann. Sc. Norm. Sup. Pisa (3) 20 (1966), 499--435.

\bibitem{Na} J.~Nash, \emph{$C^1$ isometric imbeddings}, Ann. of Math. 60 (1954), 383--396.

\bibitem{Ru} W.~Rudin, Real and Complex Analysis. Second edition. McGraw-Hill Book Co., New York-D\"usseldorf-Johannesburg, 1974.

%
%
%

\end{thebibliography}
\end{document}